\documentclass[reqno,letterpaper,11pt]{article}

\usepackage{hyperref}

\usepackage{palatino}

\usepackage{amsmath}
\usepackage{amsfonts,amssymb,amsmath,latexsym,bbm,mathrsfs}

\usepackage[all]{xy}
\usepackage[usenames]{color}
\usepackage{epsfig}

\usepackage{graphicx}
\usepackage{subcaption}

\usepackage{amsthm}
\usepackage{thmtools}

\declaretheorem[numberwithin=section]{theorem}
\declaretheorem[sibling=theorem]{proposition}
\declaretheorem[sibling=theorem]{definition}

\declaretheorem[sibling=theorem]{lemma}

\declaretheorem[sibling=theorem]{notation}

\declaretheorem[sibling=theorem,style=remark]{remark}
\declaretheorem[sibling=theorem,style=remark]{example}

\numberwithin{equation}{section}

\usepackage[mathcal]{euscript}
\usepackage{times}

\def\R{\mathbb R}
\def\C{\mathbb C}

\def\N{\mathbb N}
\def\E{\mathbb E}
\def\P{\mathbb P}
\def\bbH{\mathbb H}

\def\1{\mathbbm{1}}
\def\tensor{\otimes}

\def\M{\mathbb M}

\def\u{\mathfrak{u}}

\def\B{\mathscr{B}}
\def\A{\mathscr{A}}

\def\adots{
	\mathinner{\mkern1mu\raise1pt\hbox{.}\mkern2mu\raise4pt\hbox{.}
		\mkern2mu\raise7pt\vbox{\kern7pt\hbox{.}}\mkern1mu}}

\newcommand{\tr}{\mathrm{tr}}

\long\def\symbolfootnote[#1]#2{\begingroup%
\def\thefootnote{\fnsymbol{footnote}}\footnote[#1]{#2}\endgroup}

\begin{document}

\title{A Local Limit Theorem and Delocalization of Eigenvectors for Polynomials in Two Matrices}
\author{
Ching-Wei Ho\\
Department of Mathematics \\
Indiana University \\
Bloomington, IN 47405 \\
\texttt{cwho@iu.edu}
}

\date{\today} 

\maketitle

\begin{abstract}
We propose a boundary regularity condition for the $M_n(\C)$-valued subordination functions in free probability to prove a local limit theorem and delocalization of eigenvectors for self-adjoint polynomials in two random matrices. We prove this through estimating the pair of $M_n(\C)$-valued approximate subordination functions for the sum of two $M_n(\C)$-valued random matrices $\gamma_1\tensor C_N+\gamma_2\tensor U_N^*D_NU_N$, where $C_N$, $D_N$ are deterministic diagonal matrices, and $U_N$ is Haar unitary.
\end{abstract}

\tableofcontents

\section{Introduction}

Let $\{C_N\}_{N\in\N}$ and $\{D_N\}_{N\in\N}$ be deterministic self-adjoint diagonal $N\times N$ matrices. Suppose that $\{U_N\}_{N\in\N}$ is a sequence of $N\times N$ random unitary matrices distributed uniformly in $U(N)$. We consider random matrices $c_N=C_N$ and $d_N = U_N^*D_N U_N$. The law of $d_N$ is independent of unitary conjugation. Suppose that there are free random variables $c$, $d$ such that $c_N\to c$ and $d_N\to d$ in distribution.  

Suppose that $P$ is a self-adjoint polynomial in two noncommuting indeterminates. Our purpose is to study the density of the eigenvalues of the random matrix $X_N = P(c_N, d_N)$. In other words, given $x\in \R$ and a sequence $(\eta_N)_{N\in\N}$ that converge to $0$, we consider the ratio
$$\frac{M_{\eta_N}(x)}{2N\eta_N}$$
where $M_{\eta_N}(x)$ denotes the (random) number of eigenvalues of $X_N$ in the interval $[x-\eta_N, x+\eta_N]$. We provide technical conditions under which the limit
$$\lim_{N\to\infty}\frac{M_{\eta_N}(x)}{2N\eta_N} = \rho(x)$$
exists in probability, and the convergence is uniform for all $x$ in some closed interval $I$. The sequence $\eta_N$ is of the form $\eta_N = (cN^{-\frac{1}{12}}\log N)^\alpha$ where $c$ is a large constant and $\alpha\in(0,1)$.

Under the same technical conditions, we prove a delocalization result for the eigenvectors of $X_N$ corresponding to eigenvalues in the interval $I$. More precisely, the eigenvectors $v_1^{(N)},\ldots,v_{k_N}^{(N)}$ with the corresponding eigenvalue in $I$ satisfy
$$\P\left(\max_{j=1,\cdots k_N}\max_{i=1,\ldots,N}|v_j^{(N)}(i)|^2>N^{-\frac{\alpha}{12}}\log N\right)\leq \exp(-N^{-\delta})$$
for all sufficiently large $N$, for some $\delta>0$. Here $v_a^{(N)}(i)$ denotes the $i$-coordinate of the vector $v_a^{(N)}$.

The technical conditions we require concern regularity of a certain matrix-valued subordination function whose existence is proved in free probability theory. Our methods extend work done by Kargin \cite{Kargin2015} for the random matrix $c_N+d_N$. If $p(c_N, d_N) = c_N+d_N$ and $I$ is contained in the bulk of the eigenvalues of $X_N$ then the technical conditions are equivalent to the ones proposed in \cite{Kargin2015} and are known to be true. For general self-adjoint $p$, the technical conditions are believed to be true when the closed interval $I$ considered is inside the interior of the bulk of the $X_N$.

There have been further developments in local limit theorems regarding the model $c_N+d_N$. Bao, Erd\"{o}s and Schnelli \cite{BAO2017251, Bao2017} improved the length of the interval to the optimal scale of $N^{-1+\varepsilon}$. Discussions of the local law at the edge can be found in \cite{Bao2017SpectralRF}. Erd\"{o}s, Kr\"{u}ger and Nemish \cite{EKN} studied the local laws for polynomials of Wigner matrices.

 
 It is well-known that the eigenvalue distributions of general Wigner matrices converge weakly to the semi-circle law; this is a macroscopic phenomenon of the asymptotic behavior of the eigenvalue distributions. On a microscopic scale, Erd\"os, Schlein, and Yau \cite{ESY2009} proved a local limit theorem for general Wigner matrices, for which the number of eigenvalues in an interval of length $N^{-\frac{2}{3}}\log N$ around a point $x\in \R$ is concerned. More precisely, they prove that if $\eta^*\geq C N^{-2/3}\log N$ and $\kappa>0$, then
 $$\P\left\{\sup_{|x|\leq 2-\kappa}\left|\frac{M_{\eta^*}(x)}{2N\eta^*}-\rho_{\textrm{sc}}(x)\right|\geq\varepsilon\right\}\to 0$$
 for any $\varepsilon>0$ and sufficiently large $N$, where $\rho_{\textrm{sc}}$ is the density of the semicircle law and $M_{\eta^*}(x)$ is the number of eigenvalues of the $N\times N$ Wigner matrices in the interval $[x-\eta^*, x+\eta^*]$. They indeed showed that the decay of the probability is exponential. Note that the supremum inside the probability is taken over $[-2+\kappa, 2-\kappa]$, which does not include the boundary $\{-2, 2\}$ of the support.
 
 In addition to the eigenvalue distribution in the macroscopic and microscopic scales, one can investigate the eigenvector behaviors of the random matrix. Erd\"os, Schlein, and Yau \cite{ESY2009} proved that no eigenvector is strongly localized for a general Wigner matrix. More precisely, they proved that, given any small enough $\eta>0$ and any integer $L\geq 1$, the probability of that there is a normalized eigenvector $v$ such that the sum over the squares of any $N-L$ coordinates is bounded by $\eta$ is less than $e^{-cN}$, for some $c>0$, for all sufficiently large $N$.

Voiculescu \cite{Voiculescu1991} discovered that free probability can be used to study the limit of empirical eigenvalue distributions of random matrices. Convergence of moments and Cauchy transforms have been the main tools to study the large-$N$ limit behavior of Hermitian random matrices. Collins and Male \cite{CM2014} proved that if $c_N$ and $d_N$ are random matrices with strong free limits $c$ and $d$ (in the sense that no eigenvalue of $c_N$ (resp. $d_N$) string from the spectrum of $c$ (resp. $d$)), one of them has law independent of unitary conjugation, then, given any polynomial $P$, the sequence $P(c_N, d_N)$ converges strongly in distribution to the $P(c,d)$. In particular, the empirical eigenvalue distributions of random matrices $c_N+d_N$ converges strongly in distribution to the free additive convolution $\mu_c\boxplus \mu_d$ where $\mu_a$ denotes the law of the random variable $a$.



In order to understand properties of $P(c_N, d_N)$ for a general self-adjoint noncommutative polynomial $P$, we use the linearization technique to transform the nonlinear polynomial into a linear polynomial with matrix coefficients. We isolate some regularity properties of the subordination functions that allows us to estimate local behaviors of the empirical eigenvalue distributions of $P(c_N, d_N)$. The idea of the regularity assumption comes from the very smooth plots of the limiting distribution of $P(c,d)$, where $c$, $d$ are freely independent semicircular random variables, drawn in Belinschi, Mai and Speicher's paper \cite{BMS2013}. They plotted the graphs (\cite[Theorem 4.1]{BMS2013}) by first taking a linearization of the underlying polynomial, computing the (approximate) subordination functions as the $M_n(\C)$-valued Denjoy-Wolff point by iterations of a function at $\beta = z e_{1,1}-\gamma_0+i\eta$, where $\eta>0$ and $\Im z>0$ are small. This work suggests that the boundary of the $M_n(\C)$-valued subordination functions from the linearization is highly regular.

This paper is organized as follows. In Section 2, we include some background on free probability, operator-valued probability, and linearization. The statements of the existence of the subordination functions and the procedure of linearization to a polynomial in noncommutating variables can be found here. We also include the statement of Newton's method for functions defined on a Banach space. In Section 3, we estimate the approximate subordination functions of the $M_n(\C)$-valued unitarily invariant random matrix model; the estimates are done through Newton's method. We also show how to make use of the estimates to prove the local limit theorem and delocalization of eigenvectors for polynomials.



\section{Preliminary}
\subsection{Free Probability}
\begin{definition}
	\begin{enumerate}
		\item A $W^*$-probability space is a pair $(\A, \tau)$ where $\A$ is a von Neumann algebra and $\tau$ is a normal, faithful tracial state on $\A$. The elements in $\A$ are called (noncommuntative) random variables. 
		\item The $\ast$ - subalgebras $\A_1, \cdots \A_n\subseteq \A$ are said to be free or freely independent if, given any $i_1, i_2,\cdots, i_m\in\{1,\cdots,n\}$ with $i_k\not= i_{k+1}$, $a_{i_j}\in\A_{i_j}$ are centered, then we also have $\tau(a_{i_1}a_{i_2}\cdots a_{i_m})=0$. The random variables $a_1,\cdots, a_m$ are free or freely independent if the $\ast$-subalgebras they generate are free.
		\item For a self-adjoint element $a\in\A$, the law or distribution $\mu$ of $a$ is a probability measure on $\R$ such that whenever $f$ is a bounded continuous function, we have
		$$\int_\R f\;d\mu = \tau(f(a)).$$
	\end{enumerate}
\end{definition}

Recall that the Cauchy transform of the law $\mu$ of $a$ on the real line is given by
$$G_\mu(z) := \int_\R\frac{1}{z-t}\:\mu(dt) = \tau((z-a)^{-1})$$
for $z\not\in \textrm{supp}\:\mu$. The Cauchy transform can be defined for any finite Borel measure; however, we only use it for probability measures. The transform $G_\mu$ maps the upper half plane $\bbH^+(\C)$ into the lower half plane $\bbH^-(\C)$, and $\lim_{y\uparrow +\infty}iyG_\mu(iy) = \mu(\R)$. More results of Cauchy transform can be found in \cite{AkhiezerBook}.

The measure $\mu$ can be recovered from $G_\mu$ using the Stieltjes inversion formula, that expresses $\mu$ as a weak limit
$$\mu(dx) = \lim_{y\downarrow 0} \frac{-1}{\pi}\Im\:G_\mu(x+iy)\:dx$$
The absolutely continuous part of $\mu$ relative to Lebesgue measure is given by
$$\frac{d\mu}{dt}(x) = \lim_{y\downarrow 0} \frac{-1}{\pi}\Im\:G_\mu(x+iy)$$
and almost everywhere relative to the singular part of $\mu$,
$$\lim_{y\downarrow 0} \frac{-1}{\pi}\Im\:G_\mu(x+iy) = +\infty.$$

\subsubsection{Strong Convergence}
\begin{definition}
	Let ${\bf a} = (a_1,\ldots,a_k)$ be a $k$-tuple of random variables in the $W^*$-probability space $(\A,\tau)$. The joint distribution is the linear form $P\mapsto\tau[P({\bf a}, {\bf a}^*)]$ for all noncommutative polynomials with $2k$ noncommutative indeterminates.
	\begin{enumerate}
		\item We say that the $k$-tuples ${\bf a}_N = (a_1^{(N)},\ldots,a_k^{(N)})$, $N\in \N$, in $W^*$-probability spaces $(\A_N, \tau_N)$ converge in distribution to ${\bf a}$ if 
		$$\tau_N[P({\bf a}_N,{\bf a}_N^*)]\to\tau[P({\bf a},{\bf a}^*)]$$
		for every noncommutative polynomial $P$ in $2k$ noncommuting indeterminates.
		\item We say that the $k$-tuples ${\bf a}_N$ converge strongly in distribution if, in addition to the convergence in distribution, we also have
		$$\|P({\bf a}_N,{\bf a}_N^*)\|\to\|P({\bf a}, {\bf a}^*)\|$$
		for all noncommutative polynomials $P$ in $2k$ noncommuting indeterminates.
		\end{enumerate}
	\end{definition}
Haagerup and Thorbj\o rnsen \cite{HT2005} proved the strong asymptotic freeness of independent GUE matrices, which are $N\times N$ Hermitian matrices with independent complex Gaussian entries with variance $\frac{1}{N}$ in the upper triangular part. Thus, the theorem proved in \cite{HT2005} is an example of strong convergence.
\begin{example}[\cite{HT2005}]
	For any integer $N\geq 1$, let $X_1^{(N)}, X_2^{(N)},\ldots, X_p^{(N)}$ be $N\times N$ independent $GUE$ matrices and let $(x_1, \ldots x_p)$ be a semicircular system in a $W^*$-probability space with faithful state. Then, almost surely, for all polynomials $P$ in $p$ noncommuting indeterminates, one has
	$$\|P(X_1^{(N)},\ldots, X_p^{(N)})\|\to\|P(x_1,\ldots, x_p)\|$$
	as $N\to \infty$, where $\|\cdot\|$ are the matrix operator norm on the left hand side and the $C^*$-algebra norm on the right hand side.
	\end{example}
The above example has been generalized to other matrix models; see \cite{Anderson2013, CD2007, Male2012, Schultz2005}.

By Haar unitary matrices we mean random matrices distributed according to the Haar measure on the unitary group. A non-commutative random variable $u\in \A$ is called a Haar unitary if it is unitary and $\tau[u^n] = \delta_{n0}$. Collins and Male \cite{CM2014} proved a strong limit in distribution of independent Haar unitary matrices and (possibly random) matrices that are independent of the Haar unitary. 
\begin{proposition}[Proposition 2.1, \cite{CM2014}]
	Let ${\bf x}_N = (x_1^{(N)},\ldots,x_p^{(N)})$ and ${\bf x} = (x_1,\ldots,x_p)$ be $p$-tuples of variables in $C^*$-probability spaces $(\A_N,\tau_N)$ and $(\A,\tau)$ with faithful states. Then, the followings are equivalent:
	\begin{enumerate}
		\item ${\bf x}_N$ converges strongly in distribution to ${\bf x}$.
		\item for any continuous map $f_i, g_i:\R\to\C$, $i=1,\ldots,p$, the family of variables $(f_i(\Re x_i^{(N)}), g_i(\Im x_i^{(N)}))$ converges strongly in distribution to $(f_i(\Re x_i), g_i(\Im x_i))$.
		\item for any self-adjoint variable $h_N = P({\bf x}_N)$, where $P$ is a fixed non-commutative polynomials, $\mu_{h_N}$ converges weakly to $\mu_h$, where $h = P({\bf x})$. Moreover, the support of $\mu_{h_N}$ converges to the support of $\mu_h$ in the Hausdorff distance.
		\end{enumerate}
	In particular, the strong convergence in distribution of a single self-adjoint variable is its convergence in distribution together with the Hausdorff convergence of its spectrum.
	\end{proposition}

The random matrices that we consider are $c_N=C_N$ and $d_N = U_N D_N U_N^*$ where $C_N$, $D_N$ are self-adjoint deterministic diagonal matrices and $U_N$ is Haar distributed on the unitary group $U(N)$. The more complicated model in which $C_N$ and $D_N$ are random diagonal matrices follows from applying Fubini's theorem to the laws of the eigenvalue distributions of $c_N$ and $D_N$. This model was also considered in, for example,  \cite{BBC2018, BBCF2017, Kargin2015}.

\subsection{Operator-valued Free Probability}
\label{freeprob}

Voiculescu \cite{Voiculescu1995} introduced the concept of $W^*$-operator-valued noncommutative probability space. This is a triple $(\A, \phi, B)$ where $\A$ is a von Neumann algebra, $B\subseteq\A$ is a von Neumann subalgebra of $\A$, and $\phi:\A\to B$ is a unit-preserving conditional expectation. For us, the relevant spaces are $(M_n(\C)\tensor\A, \textrm{id}\tensor \tau, M_n(\C))$ and $(M_n(\C)\tensor (L^\infty\tensor M_n(\C)), \textrm{id}\tensor \tr_n, M_n(\C))$, where $(\A, \tau)$ is a $W^*$-probability space and $\tr_n$ is the normalized trace on matrices.

When $x\in \A$, we denote by $B\langle x\rangle$ the algebra generated by $B$ and $x$. Freeness in the operator-valued case is defined as follows.
\begin{definition}
	Two algebras $\A_1, \A_2\subseteq \A$ containing $B$ are said to be free over $B$ if
	$$\phi[x_1x_2\cdots x_n]=0$$
	whenever $n\in \N$, $x_j\in \A_{i_j}$, $\phi[x_j]=0$ and $i_j\neq i_{j+1}$, for $j=1,\ldots, n-1$. Two random variables $x$, $y\in\A$ are said to be free over $B$ if $B\langle x\rangle$ and $B\langle y\rangle$ are free over $B$.
	\end{definition}

An operator-valued (or $B$-valued) random variable $x\in\A$, the distribution $\mu_x$ is the set of multilinear maps $m_n^x:B^{n-1}\to B$ given by
$$m_n^x(b_1,\ldots,b_{n-1})=\phi[x b_1 x b_2\cdots xb_{n-1}x],\quad b_1,\ldots,b_{n-1}\in B,$$
for all $n\in\N$. When $n=0$, $m_0^x$ is the constant $1\in \B$; when $n=1$, $m_1^x=\phi[x]$.

If $x,y$ are free over $B$, then the distribution $\mu_{x+y}$ of $x+y$ depends only on the distributions $\mu_x$ and $\mu_y$. We denote this distribution by $\mu_x\boxplus \mu_y$. 

We use the $\bbH^+(B)$ to denote the upper half plane of $B$
\[\bbH^+(B) = \left\{b\in B: \Im b=\frac{b-b^*}{2i}\in B\textrm{ is positive}\right\}.\]
Similarly, $\bbH^-(B) = \left\{b\in B: \Im b = \frac{b-b^*}{2i}\in B\textrm{ is negative}\right\}$ denotes the lower half plane of $B$. The operator-valued Cauchy transform of a variable $x\in\A$ is defined by
$$G_x(b) = \phi[(b-x)^{-1}]$$
for those $b\in B$ such that $b-x$ is invertible in $\A$. If $b\in \bbH^+(B)$, then $b-x$ is invertible and $G_x(b)\in\bbH^-(B)$. The operator-valued Cauchy transform and its fully matricial extension is a very powerful tool to study operator-valued distributions. See \cite{BPV2012, Voiculescu1995, Voiculescu2000} for more details.

Suppose $(\A, \phi, B)$ is a $W^*$-operator-valued non-commutative probability space and $x, y\in \A$ are freely independent over $B$. There are two powerful tools to study the distribution $\mu_x\boxplus\mu_y$ of $x+y$. The first one is the $\mathcal{R}$-transform, introduced by Voiculescu \cite{Voiculescu1986, Voiculescu1995}. It has been well-studied also in, for example, \cite{Dykema2006, Speicher1998}. The second one is the subordination function. Its existence was proved in the scalar case in \cite{Biane1998, Voiculescu1993}, and in the operator-valued case in \cite{BMS2013, Voiculescu2000}. 
The theorem is as follows.
\begin{theorem}
	\label{Subordination}
	Assume that $(\A, \phi, B)$ is a $W^*$-operator-valued noncommutative probability space and $x,y\in \A$ are two self-adjoint operator-valued random variables that are free over $B$. Then there exists a unique pair of Fr\'echet analytic maps $\omega_1, \omega_2: \bbH^+(B)\to\bbH^+(B)$ such that
	\begin{enumerate}
		\item $\Im\omega_j(b)\geq \Im b$ for all $b\in\bbH^+(B)$, $j=1, 2$;
		\item $G_x(\omega_1(b)) = G_y(\omega_2(b)) = (\omega_1(b)+\omega_2(b)-b)^{-1}$ for all $b\in \bbH^+(B)$.
		\item $G_x(\omega_1(b)) = G_y(\omega_2(b)) = G_{x+y}(b)$ for all $b\in\bbH^+(B)$.
		\end{enumerate}
	\end{theorem}
	

\subsection{Linearization}
\label{Linear}
As in \cite{Anderson2013, BMS2013}, linearization is used to reduce a problem about a polynomial in several variables to an  addition of matrices with the variables in their entries. A \emph{linearization} of $P\in \C\langle X_1,\ldots, X_k\rangle$ is a linear polyonomial $L$ with matrix coefficients $\gamma_0,\ldots, \gamma_k$, 
$$L = \gamma_0\tensor 1+\gamma_1\tensor X_1+\ldots+\gamma_k\tensor X_k$$
satisfying the property that \emph{given $z\in\C$ and $a_1,\ldots,a_k$ in a $W^*$-probability space $\A$, $z-P(a_1,\ldots,a_k)$ is invertible in $\A$ if and only if $z e_{1,1}\tensor 1-L(a_1,\ldots,a_k)$ is invertible in $M_n(\A) = M_n(\C)\tensor \A$}, where $e_{1,1}$ is the matrix element with $1$ in the $(1,1)$-entry and $0$ elsewhere. Indeed it can be defined in a more general way -- choosing any $\alpha\in M_n(\C)$, with a different $L$, $z-P(a_1,\ldots,a_k)$ is invertible in $\A$ if and only if $z \alpha\tensor 1-L(a_1,\ldots,a_k)$ is invertible in $M_n(\A)$. Every polynomial possesses a linearization in this sense \cite{Schutzenberger1961}.

Now we describe the linearization process from \cite{Anderson2013, Mai2017}. The same procedure is also used in \cite{BBC2018} to study outliers of a polynomial in unitarily invariant random matrices. 

Given a polynomial $P\in\C\langle X_1,\ldots,X_k\rangle$, the linear polynomial $L\in M_n(\C\langle X_1,\ldots,X_k\rangle)$ is of the form
$$L=\begin{pmatrix}
0 & u\\v& Q
\end{pmatrix},$$
where $u\in M_{1,(n-1)}(\C\langle X_1,\ldots,X_k\rangle)$, $v\in M_{(n-1),1}(\C\langle X_1,\ldots,X_k\rangle)$ $Q\in M_{n-1}(\C\langle X_1,\ldots,X_k\rangle)$. The matrix $Q$ has to be invertible. The polynomial entries of $Q$, $u$ and $v$ all have degree less than or equal $1$. Moreover, $uQ^{-1}v=-P$. If $P$ is a self-adjoint polynomial, the coefficients $\gamma_i$ of $L$ can be chosen to be self-adjoint matrices.

If $P$ is a monomial of degree $0$ or $1$, we set $n=1$ and $L=P$. If $P = cX_{i_1}X_{i_2}\ldots X_{i_l}$, where $l\geq 2$ and $i_1,\ldots i_l\in \{1,\ldots,k\}$, then the matrix is of size $(l+1)\times (l+1)$ and 
$$L = \begin{pmatrix}
0& 0&\cdots&0 &0 &c\\
0& 0&\cdots& 0 & X_{i_1}&-1\\
0& 0&\cdots& X_{i_2} & -1&0\\
\vdots & \vdots&\adots& \vdots & \vdots&\vdots\\
0& X_{i_{l-1}}&\cdots& 0 & 0&0\\
X_{i_l}& -1&\cdots& 0 & 0&0\\
\end{pmatrix}$$

The lower right $l\times l$ matrix has an inverse of degree $l-1$ in $M_{l-1}(\C\langle X_1,\ldots, X_k\rangle)$ \cite{Mai2017}. Suppose now that $P = P_1+P_2$, where $P_i\in \C\langle X_1\ldots, X_k\rangle$ with linearization 
$$L_i =\begin{pmatrix}
0 & u_i\\v_i& Q_i
\end{pmatrix}\in M_{n_i}(\C\langle X_1,\ldots,X_k\rangle).$$
Then 
$$L = \begin{pmatrix}
0& u_1& u_2\\
v_1& Q_1& 0\\
v_2 & 0 & Q_2\\
\end{pmatrix}=\begin{pmatrix}
0 & u\\v& Q
\end{pmatrix}\in M_{n_1+n_2-1}(\C\langle X_1,\ldots,X_k\rangle)$$
is a linearization of $P$. Hence, every $P$ possesses a linearization.

If $P$ is a self-adjoint polynomial, then $P = P_0+P_0^*$ for some polynomial $P_0$. Suppose 
$$L_0=\begin{pmatrix}
0 & u_0\\v_0& Q_0
\end{pmatrix}$$
is a linearization of $P_0$, then
$$L_i =\begin{pmatrix}
0 & u_i\\v_i& Q_i
\end{pmatrix}\in M_{n_i}(\C\langle X_1,\ldots,X_k\rangle).$$
Then 
$$L = \begin{pmatrix}
0& u_0& v_0^*\\
u_0^*& 0& Q_0^*\\
v_0 & Q_0 & 0\\
\end{pmatrix}=\begin{pmatrix}
0 & u\\u^*& Q
\end{pmatrix}$$
is a self-adjoint linear polynomial for $P$. The constant term of $Q^{-1}$ has spectrum contained in $\{1,-1\}$ \cite{Mai2017}.

More properties of this linearization process were proved in \cite[Section 4]{BBC2018}. The following lemma, which gives an estimate of norm between the polynomial $P$ and the linearization $L$, is of particular interest to us.

\begin{lemma}[Lemma 4.3 \cite{BBC2018}]
	\label{Linv}
	Suppose that $P=P^*\in \C\langle X_1,\ldots,X_k\rangle$, and let $L$ be a linearization of $P$ constructed with the properties above. There exist polynomials $T_1$, $T_2$ with nonnegative coefficients with the following property:
	given arbitrary selfadjoint elements $S_1,\ldots,S_k$ in a unital $C^*$-algebra $\A$, and given $z_0\in \C$ such that $z_0-P(S)$ is invertible, we have
	$$\|(z_0 e_{1,1} - L(S))^{-1}\|\leq T_1(\|S_1\|,\ldots,\|S_k\|)\|(z_0-P(S))^{-1}\|+T_2(\|S_1\|,\ldots,\|S_k\|).$$
	In particular, given constants $C, \delta>0$, there exists $\varepsilon>0$ such that $\textrm{dist}(z_0,\sigma(P(S)))\geq \delta$ and $\|S_1\|+\ldots+\|S_k\|\leq C$ imply $\textrm{dist}(0,\sigma(z_0 e_{1,1}-L(S)))\geq \varepsilon$.
	\end{lemma}

\subsection{Newton's Method}
Newton's Method provides a way to locate a zero from an initial guess; it also gives an estimate about the distance between the zero and the initial guess. It has been generalized into different setups to fit the needs of different applications. The one that we need is to locate a zero of a nonlinear operator on a Banach space.

Let $F$ be a nonlinear operator on a domain $G$ in a Banach space $X$. To use Newton's method to locate the solution $F(x) = 0$, we first make an initial guess $y_0$. If $F$ is twice differentiable, and $F'$ is invertible in a neighborhood of $y_0$, then the iterations
$$y_{n+1} = y_n - [F'(y_n)]^{-1}F(y_n)$$
converge as $n\to \infty$ to the solution of $F(x)=0$ under the sufficient conditions of the following theorem due to Kantorovich \cite{Kantorovich1948}.
\begin{theorem}[Kantorovich]
	Assume, in the framework described above, there are constants $C_0$, $\delta_0$ and $A$ such that the following conditions hold:
	\begin{enumerate}
		\item The operator $F'(y_0)$ is invertible whose inverse has norm $\|[F'(y_0)]^{-1}\|\leq C_0$;
		\item $\|[F'(y_0)]^{-1}F(y_0)\|\leq \delta_0$;
		\item the second derivative $F''(y)$, $y\in G$ is bounded in $G$ by $A$, namely $\|F''(y)\|\leq A$;
		\item the constants $C_0$, $\delta_0$, $A$ satisfy the inequality $h_0 = C_0\delta_0A\leq \frac{1}{2}$.
		\item $B\left(y_0, \frac{1-\sqrt{1-2h_0}}{h_0}\delta_0\right)\subseteq G.$
	\end{enumerate}
	Then the equation $F(w)=0$ has a unique solution $w^*$ in the ball $B\left(y_0, \frac{1-\sqrt{1-2h_0}}{h_0}\delta_0\right)$ centered at our initial approximation $w_0$ and the iterations of the Newton's method $w_n\to w^*$ as $n\to \infty$.
\end{theorem}

\section{The Unitarily Invariant Model}
\subsection{The Sum of Two $M_n(\C)$-valued Matrices}
\label{Uinv}
Fix two sequences $\{C_N\}_{N\in\N}$ and $\{D_N\}_{N\in\N}$ of deterministic diagonal matrices $C_N, D_N\in M_N(\R)$. Suppose that both sequences have a strong limiting (deterministic) empirical eigenvalue distribution. Let $\{U_N\}_{N\in\N}$ be a sequence of $N\times N$ unitary matrices distributed uniformly on $U(N)$. 

We view $d_N = U_N D_N U_N^\ast$ and $c_N = C_N$ as random matrices in the space $M_N(\C)\tensor L^\infty$. Let $(\A, \tau)$ be a $W^*$-probability space and let $c,d\in\A$ be freely independent self-adjoint variables such that $(c_N, d_N)\to (c, d)$ strongly in distribution. The $M_n(\C)$-valued probability spaces $(M_n(\C)\tensor (M_N(\C)\tensor L^\infty), M_n(\E\tr_N), M_n(\C))$, where $\E$ is the expectation, fit in the framework of operator-valued noncommutative probability space introduced in Section \ref{freeprob}. We look at the ``limit" $M_n(\C)$-valued probability space $(M_n(\A), M_n(\tau), M_n(\C))$.

We focus on the $M_n(\C)$-valued unitarily invariant matrix model
$$H_N= \gamma_1\tensor c_N +  \gamma_2\tensor d_N\in M_n(\A),$$
where $\gamma_1, \gamma_2$ are Hermitian matrices in $M_n(\C)$. (Later we apply the results in this section to a linearization of a noncommutative polynomial $P(c_N, d_N)$.)

\begin{notation}
	We use the following notations throughout the rest of this paper:
	\begin{enumerate}
		\item $\E$ is the expectation of a random variable (in a probability space). We also use $\E$ to mean taking expectation entrywise in a matrix; $\E$ is completely positive.
		\item $H_N = \gamma_1 \tensor c_N + \gamma_2\tensor d_N$; and $H = \gamma_1\tensor c+\gamma_2\tensor d$.
		\item $R_N(\beta) = (\beta\tensor I_N - H_N)^{-1}$ is the resolvent of $H_N$, where $\beta\in \bbH^+(M_n(\C))$. 
		\item The (random $M_n(\C)$-valued) Cauchy transform $G_{H_N}(\beta) = M_n(\tr_N)(R_N(\beta))$. We also write the (random) Cauchy transform $G_{c_N}(\beta) = M_n(\tr_N)((\beta\tensor I_N-\gamma_1\tensor c_N)^{-1})$ and a similar definition for $d_N$.
		\item We write $G_H(\beta) = M_n(\tau)[(\beta\tensor I_N-H)^{-1}]$, $\beta\in \bbH^+(M_n(\C))$, and $R(\beta) = (\beta\tensor I_N-H)^{-1}$. We sometimes write $\beta$ in place of $\beta\tensor I_N$ when context is clear.
		\item We define approximate subordination functions, for $\beta\in\bbH^+(M_n(\C))$ by
			\begin{itemize}
				\item $\omega_{1}^{(N)}(\beta) = \beta - (\E G_{H_N}(\beta))^{-1}(\E f_{1}^{(N)}(\beta))$
				\item $\omega_{2}^{(N)}(\beta) = \beta - (\E G_{H_N}(\beta))^{-1}(\E f_{2}^{(N)}(\beta))$
				\end{itemize}
			where $f_{1}^{(N)}(\beta) = M_n(\tr_N)(R_N(\beta)\cdot (\gamma_2\tensor d_N))$ and $f_{2}^{(N)}(\beta) = M_n(\tr_N)(R_N(\beta)\cdot (\gamma_1\tensor c_N))$.
		\end{enumerate}
	\end{notation}

These are analogs of definitions from \cite{Kargin2015}. We observe that
$$f_{1}^{(N)}(\beta) + f_{2}^{(N)}(\beta) = M_n(\tr_N)((H_N-\beta) R_N(\beta))+M_n(\tr_N)(\beta R_N(\beta)) = -I_N+\beta G_{H_N}(\beta)$$
and $$\omega_{1}^{(N)}(\beta)+\omega_{2}^{(N)}(\beta) =\beta + (\E G_{H_N}(\beta))^{-1}.$$
The approximate subordination functions satisfy the same equation as the subordination functions of the free random variables in the limit, as $N\to\infty$. We will show that these functions are ``almost" self-map of the upper half plane $\bbH^+(M_n(\C))$ (See Proposition \ref{Impart}) and that they converge to the subordination functions of the free random variables (See equation \eqref{OmegaExp} and \cite[Proposition 8.8]{BBC2018}). An explicit estimate of these approximate subordination functions is given in Proposition \ref{Newton}. We use them when we approximate the Cauchy transform $G_{H_N}(\beta)$ when $\beta$ is very close to the boundary $\partial \bbH^+(M_n(\C))$ of the upper half plane $\bbH^+(M_n(\C))$.

From now on, we simply write $\tr$ instead of $\tr_N$ for the normalized trace for $N\times N$ matrices.

\begin{lemma} We have
	\label{Id1}
	$$\E(G_{H_N}(\beta)\tensor I_N\cdot\gamma_2\tensor d_N\cdot R_N(\beta)) 
	=\E[f_{1}^{(N)}(\beta) R_N(\beta)].$$
	\end{lemma}
\begin{proof}
	By \cite[Proof of Lemma 8.1]{BBC2018}, for any $Z\in M_N(\C)$,
	\begin{equation}
		\label{commut}
		\E(R_N(\beta)[I_n\tensor Z,\gamma_2\tensor d_N]R_N(\beta)) = 0
		\end{equation}
	where the $[\cdot,\cdot]$ is the commutator.
	
	The matrix $R_N(\beta)\in M_{nN}(\C)\tensor L^\infty$ is of the form $\sum_{j_1,j_2=1}^n e_{j_1,j_2}\tensor A_{j_1,j_2}$, where $e_{j_1, j_2}$ is the $n\times n$ matrix with $1$ in the $(j_1,j_2)$-entry and $0$ elsewhere. We take $Z = e_{ab}$. Since $(Zd_N)_{xy} = \delta_{xa} (d_N)_{by}$ and $(d_N Z)_{xy} = \delta_{yb}(d_N)_{xa}$, the $(u,v)$-entries of $A_j(Zd_N)A_k$ and $A_j(d_N Z)A_k$ are $(A_j)_{ua}(d_N A_k)_{bv}$ and $(A_j d_N)_{ua}(A_k)_{bv}$ respectively. 
	Expanding the matrix multiplication gives
	$$R_N(\beta)\cdot \gamma_2\tensor (Z d_N)\cdot R_N(\beta) = \sum_{j_1, j_2,k_1,k_2}e_{j_1,j_2}\,\gamma_2\,e_{k_1,k_2} \tensor A_{j_1,j_2} (Z d_N) A_{k_1,k_2}$$
	and similarly 
	$$R_N(\beta)\cdot \gamma_2\tensor (d_N Z)\cdot R_N(\beta) = \sum_{j_1,j_2,k_1,k_2}e_{j_1,j_2}\,\gamma_2\,e_{k_1,k_2} \tensor A_{j_1,j_2}  (d_N Z) A_{k_1,k_2}.$$
	By \eqref{commut}, the expectation of above two quantities are equal, so
	$$\E\left[\sum_{j_1,j_2,k_1,k_2} e_{j_1,j_2}\,\gamma_2\,e_{k_1,k_2}\cdot (A_{j_1,j_2})_{ua}(d_NA_{k_1,k_2})_{bv}\right] = \E\left[\sum_{j_1,j_2,k_1,k_2}e_{j_1,j_2}\,\gamma_2\,e_{k_1,k_2}\cdot (A_{j_1,j_2}d_N)_{ua}(A_{k_1,k_2})_{bv}\right].$$
	We first take $u=a$, then sum over all $a=1,2,\ldots,N$ and divide by $N$; we have
	$$\E\left[\sum_{j_1,j_2,k_1,k_2} e_{j_1,j_2}\,\gamma_2\,e_{k_1,k_2}\cdot \tr(A_{j_1,j_2})(d_NA_{k_1,k_2})_{bv}\right] = \E\left[\sum_{j_1,j_2,k_1,k_2}\,e_{j_1,j_2}\,\gamma_2\,e_{k_1,k_2}\cdot \tr(A_{j_1,j_2}d_N)(A_{k_1,k_2})_{bv}\right].$$
	Since this this is true for all $b$, $v$, we see that
	\begin{equation}
	\begin{split}
		&\E\left[\sum_{j_1,j_2,k_1,k_2}e_{j_1,j_2}\,\gamma_2\,e_{k_1,k_2}\tensor \tr(A_{j_1,j_2})(d_N A_{k_1,k_2})\right] \\
		= &\E\left[\sum_{j_1,j_2,k_1,k_2} e_{j_1,j_2}\,\gamma_2\,e_{k_1,k_2}\tensor \tr(A_{j_1,j_2}d_N)A_{k_1,k_2}\right]. \label{commut2}
		\end{split}
	\end{equation}
	Thus
	\begin{align*}
		\E(G_{H_N}(\beta)\,\gamma_2\tensor d_N\,R_N(\beta)) &= \E \left[\sum_{j_1,j_2} e_{j_1,j_2} \tensor (\tr(A_{j_1,j_2})I_N)\cdot \gamma_2\tensor d_N\cdot \sum_{k_1,k_2} e_{k_1,k_2}\tensor A_{k_1,k_2}\right]\\
		&= \E\left[\sum_{j_1,j_2,k_1,k_2}e_{j_1,j_2}\,\gamma_2\,e_{k_1,k_2} \tensor \tr(A_{j_1,j_2})(d_NA_{k_1,k_2})\right]\\
		&=\E\left[\sum_{j_1,j_2,k_1,k_2} e_{j_1,j_2}\, \gamma_2\,e_{k_1,k_2}\tensor \tr(A_{j_1,j_2} d_N)A_{k_1,k_2}\right]\qquad\qquad\textrm{by \eqref{commut2}}\\
		&=\E\left[\left(\sum_{j_1,j_2} e_{j_1,j_2}\gamma_2\tensor \tr(A_{j_1,j_2}d_N)I_N\right)\left( \sum_{k_1,k_2} e_{k_1,k_2}\tensor A_{k_1,k_2}\right) \right]\\
		&= \E \left[M_n(\tr)(R_N(\beta)\cdot \gamma_2\tensor d_N)\tensor I_N \cdot R_N(\beta)\right]\\
		&=\E[f_{1}^{(N)}(\beta) R_N(\beta)],
		\end{align*}
		and the lemma follows.
	\end{proof}

The subordination function $\omega_1$ is given by
$$\omega_1(\beta) \tensor 1= M_n(\tau)[R(\beta)]^{-1}+\gamma_1\tensor c.$$
To imitate the behavior of $\omega_1$ in the case of finite dimensional matrices, we try to look at $\E[R_N(\beta)]^{-1}+\gamma_1\tensor c_N$. However, this function is no longer of the form $M_n(\C)\tensor I_N$ because $c_N$ and $d_N$ are not free.

We next estimate how far $\omega_{1}^{(N)}$ is from $\E[R_N(\beta)]^{-1}+\gamma_1\tensor c_N$. The error term $\delta_1(\beta)$ is defined to be
\begin{equation}
\label{delta1}
\delta_1(\beta) = \E[G_{H_N}(\beta)]^{-1}\E\Delta_1(\beta)
\end{equation}
where
\begin{equation}
\label{Delta}
\Delta_1(\beta) = \E[(f_{1}^{(N)}-\E f_{1}^{(N)})\,R_N(\beta)]-\E[(G_{H_N}(\beta)-\E G_{H_N}(\beta))\,(\beta-\gamma_1\tensor c_N)\,R_N(\beta)].
\end{equation}
The definition of $\Delta_1$ allows us to approximate by applying concentration inequality to $f_{1}^{(N)}-\E f_{1}^{(N)}$ and $G_{H_N}(\beta)-\E G_{H_N}(\beta)$. We also define $\delta_2$ and $\Delta_2$ analoguously, by replacing $\gamma_1\tensor c_N$ by $\gamma_2\tensor d_N$ and $f_1^{(N)}$ by $f_2^{(N)}$. 
We use the same notation as above.
\begin{proposition}
	\label{error}
	We have
	\begin{equation}
		\label{remainder}
		(\omega_{1}^{(N)}(\beta)-\gamma_1\tensor c_N)\E[R_N(\beta)] = 1 + \delta_1(\beta).
		\end{equation}
\end{proposition}
\begin{remark}
	We will prove that under extra assumptions in Proposition \ref{Impart}, $\omega_{1}^{(N)}(\beta)-\gamma_1\tensor c_N$ is indeed invertible, and $\omega_{1}^{(N)}(\beta)$ is in $\bbH^+(M_n(\C))$.
\end{remark}
\begin{proof}
	By the identity $(\beta-\gamma_1\tensor c_N)R_N(\beta) = 1+(\gamma_2\tensor d_N)R_N(\beta)$, multiplying $G_H(\beta)\tensor I_N$ and taking expectation give us
	\begin{align*}
		\E[G_{H_N}(\beta)\tensor I_N \cdot(\beta-\gamma_1\tensor c_N)\cdot R_N(\beta)] &= \E[G_{H_N}(\beta)\tensor I_N] + \E[G_{H_N}(\beta)\tensor I_N\cdot \gamma_2\tensor d_N\cdot R_N(\beta)]\\
		&= \E[G_{H_N}(\beta)\tensor I_N] + \E[f_{1}^{(N)}\tensor I_N \cdot R_N(\beta)]
	\end{align*}
where the last equality follows from Lemma \ref{Id1}. Now the left hand side is $$\E[(G_{H_N}(\beta)-\E[G_{H_N}(\beta)])(\beta-\gamma_1\tensor c_N)R_N(\beta)]+\E[G_{H_N}(\beta)]\E[(\beta-\gamma_1\tensor c_N)R_N(\beta)]$$ while the right hand side is $$\E[G_{H_N}(\beta)]+\E[(f_{1}^{(N)}(\beta)-\E[f_{1}^{(N)}(\beta)])R_N(\beta)]+\E[f_{1}^{(N)}(\beta)]\E[R_N(\beta)].$$ Rearranging the terms gives
\begin{align*}
	&\E[G_{H_N}(\beta)]\E[(\beta-\gamma_1\tensor c_N-(\E G_{H_N}(\beta))^{-1}(\E f_{d_N}(\beta)))\cdot R_N(\beta)]\\
	=&\E[G_{H_N}(\beta)]+\E[(f_{1}^{(N)}(\beta)-\E f_{1}^{(N)}(\beta))\cdot R_N(\beta)]-\E[(G_{H_N}(\beta)-\E G_{H_N}(\beta))\cdot(\beta-\gamma_1\tensor c_N)\cdot R_N(\beta)].
	\end{align*}
	Recall that $\omega_{1}^{(N)}(\beta) = \beta - (\E G_{H_N}(\beta))^{-1}(\E f_{1}^{(N)}(\beta))$; it follows that
	$$(\omega_{1}^{(N)}(\beta)-\gamma_1\tensor c_N)\E[R_N(\beta)] = 1 + \delta_1(\beta)$$
	where $\Delta_1(\beta)$ is defined in \eqref{Delta}.
	\end{proof}
	
	Assume that $\eta \in (0,1)$ and $z\in K$ for some bounded set $K\subseteq \bbH^+(\C)$. Using linearization technique, we are particularly interested in $\beta = z e_{1,1}-\gamma_0+i\eta$. Its real part is bounded uniformly for all $z\in K$ and $\|(\Im\beta)^{-1}\| = \frac{1}{\eta}$ since $z\in \bbH^+(\C)$. We look more generally at $\{\beta\in\bbH^+(M_n(\C)): \|\beta\|\leq K, \|(\Im\beta)^{-1}\|=\frac{1}{\eta}\}$ where $\eta\in(0,1)$, and $K>0$.
	\begin{lemma}
		\label{DeltaEst}
		Suppose that $\beta\in\bbH^+(M_n(\C))$, $ \|\beta\|\leq K$, and $\|(\Im\beta)^{-1}\|=\frac{1}{\eta}>1$. Then
		$$\|\E\Delta(\beta)\| \leq C\left(\frac{1}{N^{\frac{1}{2}}\eta^3}\right)$$
		where $C$ depends only on $K$, $\gamma_1$, $\gamma_2$, $\|c\|$, $\|d\|$.
		\end{lemma}
	\begin{proof}
		
		By \cite[Equation 13]{BPV2012}, 
		$$\|R_N(\beta)\|\leq \|(\Im\beta)^{-1}\| = \frac{1}{\eta}.$$
		We estimate $\Delta$ term-by-term. For the second term $\E[(G_{H_N}(\beta)-\E G_{H_N}(\beta))\,(\beta-\gamma_1\tensor c_N)\,R_N(\beta)]$, we have
		\begin{align*}
		\|\E[(G_{H_N}(\beta)-\E G_{H_N}(\beta))\,(\beta-\gamma_1\tensor c_N)\,R_N(\beta)]\|\leq & \|\E[(G_{H_N}(\beta)-\E G_{H_N}(\beta))\|\|(\beta-\gamma_1\tensor c_N)\|\|R_N(\beta)\|\\
		\leq & \frac{C(K, \gamma_1, \gamma_2, c_N, d_N)}{\eta}\|G_{H_N}(\beta) - \E G_{H_N}(\beta)\|.
		\end{align*}
		 By concentration inequality, and by the fact that $G_{H_N}(\beta)-\E G_{H_N}(\beta) = M_n(\tr)(R_N(\beta)-\E(R_N(\beta)))$ has rank bounded by $n$ which is independent of $N$, we obtain, as in \cite[Proof of Proposition 8.4]{BBC2018}
		 $$\P\left(\|G_{H_N}(\beta)-\E G_{H_N}(\beta)\|>\varepsilon\right)\leq 2\exp\left(\frac{-N\varepsilon^2}{8r^4 \|\gamma_2\|^2\|d_N\|^2}\right)$$
		 for any $\varepsilon > 0$ and $\alpha\in (0,\frac{1}{2})$, where $r=\sup_N\|R_N(\beta)\|$ (The norm is taken over $M_n(\C)\tensor L^\infty$, which is a positive constant, \emph{not} a random norm). Since $r = O(1/\eta)$, we have
		 $$\P\left(\|(G_{H_N}(\beta)-\E G_{H_N}(\beta))(\beta-\gamma_1\tensor c_N)R_N(\beta)\|>\varepsilon\right)\leq 2\exp\left(\frac{-c \varepsilon^2N\eta^6}{\|\gamma_2\|^2\|d_N\|^2}\right).$$
		 
		 For the first term $(f_{1}^{(N)}-\E f_{1}^{(N)})R_N(\beta)$, since $f_{d_N}(\beta) = -I_n + M_n(\tr)((\beta-\gamma_1\tensor c_N)R_N(\beta))$, we have
		 \begin{align*}
		 (f_{d_N}-\E f_{d_N}) =& M_n(\tr)\big((\beta-\gamma\tensor c_N)(R_N(\beta)-\E R_N(\beta))\big) \\
		 =& \beta M_n(\tr)[R_N(\beta)-\E R_N(\beta)]-M_n(\tr)[(\gamma_1\tensor c_N)(R_N(\beta)-\E R_N(\beta))].
		 \end{align*}
		 The terms $M_n(\tr)[R_N(\beta)-\E R_N(\beta)]$ and $M_n(\tr)[(\gamma_1\tensor c_N)(R_N(\beta)-\E R_N(\beta))]$ are linear transformations of $R_N(\beta)-\E R_N(\beta)$ with at most rank $n$. Hence,
		 $$\P\left(\|(f_{1}^{(N)}-\E f_{1}^{(N)})R_N(\beta)\|>\varepsilon\right)\leq 2\exp\left(\frac{-C \varepsilon^2 N\eta^6}{\|\gamma_2\|^2\|d_N\|^2}\right).$$
		 Using triangle inequality, we have
		 $$\P(\|\Delta_1\|\geq \varepsilon)\leq 4\exp\left(\frac{-C N\eta^6 \varepsilon^2}{\|\gamma_2\|^2\|d_N\|^2}\right).$$
		 
		 Thus
		 $$\|\E \Delta_1(\beta)\|\leq \E\|\Delta_1(\beta)\|\leq \int_0^\infty 4\exp\left(\frac{-C N\eta^6 t^2}{\|\gamma_2\|^2\|d_N\|^2}\right) dt = O\left(\frac{1}{N^{\frac{1}{2}\eta^3}}\right),$$
		 the lemma is proved.
		\end{proof}
	
	We now would like to consider the map $\beta\mapsto \E[G_{H_N}(\beta)]^{-1}$ for $\beta\in \bbH^+(M_n(\C))$. Analogue to \cite{BBCF2017} where the completely positive map $\E$ is considered, we consider 
	\begin{align*}
	\widetilde{F}(\varepsilon) = &\left((M_n(\tr)\E)[(\Im\beta-\varepsilon(H_N-\Re\beta)^{-1}]\right)^{-1}\\
	= &z\:\Im\beta+\Re\beta-M_n(\tr)[\E(H_N)]\\
	&-\frac{1}{z}(M_n(\tr)\E)[(H_N-(M_n(\tr)\E)(H_N))(\Im\beta^{-1})(H_N-(M_n(\tr)\E)(H_N))]+O\left(\frac{1}{z}\right).
	\end{align*}
	Hence, \begin{align*}
	F(z) = \left((M_n(\tr)\E)[(\Re\beta+z\Im\beta-H_N)^{-1}]\right)^{-1} =A+Bz-\int_\R\frac{\rho(dt)}{z-t}
	\end{align*}
	where $A = \Re\beta-(M_n(\tr)\E)(H_N)$, $B = \Im\beta$ and
	$$\rho(\R) = \lim_{z\to\infty} z(A+Bz-F(z)) = (M_n(\tr)\E)[(H_N-(M_n(\tr)\E)(H_N))(\Im\beta)^{-1}(H_N-(M_n(\tr)\E)(H_N))].$$
	It follows that
	$$\|(\E G_{H_N}(\beta))^{-1}\|\leq \|\beta\|+\|M_n(\tr)\E(H_N)\|+\|\rho(\R)\| = O\left(\frac{1}{\eta}\right)$$
	where constant of the $O$ in last equality is for $\Im\beta \geq \eta$, with $\eta\in(0,1)$ and depends on $\|\beta\|\leq K$.

	Hence, if $\|\beta\|\leq K$ and $\Im\beta\geq \eta\in(0,1)$, Lemma \ref{DeltaEst} shows that
	\begin{equation}
	\label{Eerror}
	\|(\E G_{H_N}(\beta))^{-1}\E\Delta_1(\beta)\| = O\left(\frac{1}{N^{\frac{1}{2}}\eta^4}\right).
	\end{equation}
	
	\bigskip

	The preceding observation shows that $\omega_{1}^{(N)}(\beta)-\gamma_1\tensor c_N$ is invertible as long as $\|(\E G_{H_N}(\beta))^{-1}\E\Delta_1(\beta)\|$ is small, because the invertible matrices form an open set. When $\|(\E G_H(\beta))^{-1}\E\Delta_1(\beta)\|$ is small, we can write \eqref{remainder} as
	$$\E[R_N(\beta)] = (\omega_{1}^{(N)}(\beta)-\gamma_1\tensor c_N)^{-1}\,\left(I + \E[G_{H_N}(\beta)]^{-1}\E[\Delta_1(\beta)]\right).$$
	Hence $\E[R_N(\beta)]^{-1} = \left(I + \E[G_{H_N}(\beta)]^{-1}\E[\Delta_1(\beta)]\right)^{-1} (\omega_{1}^{(N)}(\beta)-\gamma_1\tensor c_N)$ which shows
	\begin{equation}
	\label{OmegaExp}
	\omega_{1}^{(N)}(\beta) = \E[R_N(\beta)]^{-1} + \gamma_1\tensor c_N - \left[(I+\E[G_{H_N}(\beta)]^{-1}\E\Delta(\beta))^{-1}-I\right](\omega_{1}^{(N)}(\beta)-\gamma_1\tensor c_N).
	\end{equation}
	
	The above formula shows the difference between this definition of $\omega$ and the choice from \cite{BBC2018}. This formula will be used as a measure of how far $\omega$ is from being a self-map on the upper half plane of $M_n(\C)$. Then, we estimate the difference of $\omega_1^{(N)}$ from $\E[R_N(\beta)]^{-1} + \gamma_1\tensor c_N$. The advantages of our choice are that the $\omega$'s are already in $M_n(\C)$ which allows us to estimate the Cauchy transforms $G_{c_N}(\omega_{1}^{(N)}(\beta))$ and $G_{d_N}(\omega_{2}^{(N)}(\beta))$; that the functions $\omega_1^{(N)}$ and $\omega_2^{(N)}$ satisfy the same equation as the subordination functions also allows us to apply Newton's method.
	
	The following proposition gives an estimate of the imaginary part of $\omega_{1}^{(N)}(\beta)$, in addition to the invertibility of $\omega_{1}^{(N)}(\beta)-\gamma_1\tensor c_N$.
	\begin{proposition}
		\label{Impart}
		Suppose $\beta\in \bbH^+(\C)$, $\|\beta\|\leq K$ and $\|(\Im\beta)^{-1}\|=\frac{1}{\eta}$, $\eta\in(0,1)$.
		If $N^{\frac{1}{2}}\eta^4$ is large enough, then
		$$\Im(\omega_{1}^{(N)}(\beta))\geq \Im\beta - \frac{c}{N^{\frac{1}{2}}\eta^6}$$
		for some positive constant $c$. In particular, if $N^{\frac{1}{2}}\eta^6$ is large enough, $\omega_{1}^{(N)}(\beta)\in \bbH^+(M_n(\C))$.
		\end{proposition}
	\begin{proof}
		Suppose $N^{\frac{1}{2}}\eta^4$ is large enough such that $\|(\E G_H(\beta))^{-1}\E \Delta_1(\beta)\|<\frac{1}{2}$. By power series expansion, we know that if $\|X\|\leq \varepsilon<\frac{1}{2}$, then $\|(I+X)^{-1}-I\|\leq 2\varepsilon$. We have
			$$\|(I+(\E G_H(\beta))^{-1}\E\Delta_1(\beta))^{-1}-I\| = O\left(\frac{1}{N^{\frac{1}{2}}\eta^4}\right).$$
			Recall that $f_{1}^{(N)}(\beta) = M_n(\tr)(R_N(\beta)\cdot \gamma_2\tensor d_N)$. We have the estimate
			$$\|\E f_{1}^{(N)}(\beta)\|\leq \E\|f_{1}^{(N)}(\beta)\|\leq \|\gamma_2\tensor d_N\|\cdot\E\|R_N(\beta)\| = O\left(\frac{1}{\eta}\right).$$
			Hence if $\beta\in \bbH^+(M_n(\C))$, $\|\beta\|\leq K$ and $\|(\Im\beta)^{-1}\|=\frac{1}{\eta}$, $\eta\in(0,1)$, then
			$$\|\omega_{1}^{(N)}(\beta)\|=\|\beta-(\E G_{H_N}(\beta))^{-1}\E f_{1}^{(N)}(\beta)\|=O\left(\frac{1}{\eta^2}\right).$$

			It follows that 
			$$\|((I+(\E G_{H_N}(\beta))^{-1}\E\Delta_1(\beta))^{-1}-I)(\omega_{1}^{(N)}(\beta)-\gamma_1\tensor c_N)\| = O\left(\frac{1}{N^{\frac{1}{2}}\eta^6}\right).$$
			
			Since $\Im([\E R_N(\beta)]^{-1})\geq \Im(\beta)$, we have, by looking at the expression \eqref{OmegaExp},
			$$\Im(\omega_{1}^{(N)}(\beta))\geq \Im\beta - \frac{c}{N^{\frac{1}{2}}\eta^6}.$$
			for some positive constant $c$.
			
			If $N^{\frac{1}{2}}\eta^6$ large enough, $\omega_{1}^{(N)}(\beta)\in \bbH^+(M_n(\C))$ because $\Im\beta\geq \varepsilon > 0$ for some $\varepsilon$.
		\end{proof}

	\begin{theorem}
		\label{FinEst}
		Suppose $\|\beta\|\leq K$,  $\beta\in \bbH^+(\C)$, $\|(\Im\beta)^{-1}\|=\frac{1}{\eta}$, $\eta\in(0,1)$.  If $N^{\frac{1}{2}}\eta^4$ is large enough, then the following estimates hold:
		\begin{enumerate}
			\item $\|\E R_N(\beta)-(\omega_{1}^{(N)}(\beta)-\gamma_1\tensor c_N)^{-1}\|=O\left(\frac{1}{N^{\frac{1}{2}}\eta^5}\right).$
			\item $\|\E G_H(\beta)-M_n(\tr)[(\omega_{1}^{(N)}(\beta)-\gamma_1\tensor c_N)^{-1}]\|=O\left(\frac{1}{N^{\frac{1}{2}}\eta^5}\right).$
			\end{enumerate}
		\end{theorem}
	\begin{proof}
		Equation $\eqref{remainder}$ says
		$$(\omega_{1}^{(N)}(\beta)-\gamma_1\tensor c_N)^{-1}=\E[R_N(\beta)] \left(I + \E[G_H(\beta)]^{-1}\E[\Delta(\beta)]\right)^{-1}.$$
		By Equation \eqref{Eerror}, $\|(\E G_{H_N}(\beta))^{-1}\E\Delta_1(\beta)\| = O\left(\frac{1}{N^{\frac{1}{2}}\eta^4}\right)$. If $\|\E[G_{H_N}(\beta)]^{-1}\E[\Delta_1(\beta)]\|<\frac{1}{2}$, then using power series expansion to $\left(I + \E[G_{H_N}(\beta)]^{-1}\E[\Delta_1(\beta)]\right)^{-1}$, we have
		$$\|(\omega_{1}^{(N)}(\beta)-\gamma_1\tensor c_N)^{-1}\| = O\left(\frac{1}{\eta}\right).$$
		It follows that the error term $\delta_1(\beta)$ from proposition \ref{error} satisfies
		$$\|\E R_N(\beta)-(\omega_{1}^{(N)}(\beta)-\gamma_1\tensor c_N)^{-1}\|\leq  \|(\omega_{1}^{(N)}(\beta)-\gamma_1\tensor c_N)^{-1}\|\|\E[G_{H_N}(\beta)]^{-1}\E \Delta_1(\beta)\|=O\left(\frac{1}{N^{\frac{1}{2}}\eta^5}\right),$$
		which is the first assertion.
		The second assertion follows directly from $M_n(\tr)$ is continuous.
		\end{proof}

			We now extend the notion of smoothness in \cite{Kargin2015}  to $M_n(\C)$-valued case. The following reduces to the smoothness condition in \cite{Kargin2015} when $n=1$.

		\begin{definition}
			\label{reg}
			Suppose that $c$, $d$ are freely independent random variables, with corresponding distributions $\mu_c$ and $\mu_d$, and $\gamma_1, \gamma_2\in\M_n(\C)$ are self-adjoint matrices. Let
			$$H = \gamma_1\tensor c + \gamma_2\tensor d\in M_n(\A)$$
			be a sum of $M_n(\C)$-valued free variables, equipped with the conditional expectation $M_n(\tau)$. 
			
			We say that the triple $(\mu_c, \mu_d, H)$ is regular on $(S,V)$, where $S\subseteq \partial\bbH^+(M_n(\C))$, $V$ is a subspace of Hermitian matrices,  if the restrictions of $\omega_1$ and $\omega_2$ to $\{x+iy\in\bbH^+(M_n(\C)): x\in S, y\in V\}$ extend continuously to $\partial\{x+iy\in\bbH^+(M_n(\C)): x\in \bar{S}, y\in V\}$ and if the triple $(\mu_c, \mu_d, H)$ satisfies the following two conditions:
				\begin{enumerate}
				\item[A.] For all $\bar{\beta}\in \bar{S}$, $\omega_1(\bar{\beta})-t\gamma_1$ is invertible for every $t\in \textrm{supp}\,\mu_c$ and $\omega_2(\bar{\beta})-t\gamma_2$ is invertible for every $t\in \textrm{supp}\,\mu_d$;
				\item[B.] For each $(w_1,w_2):=(\omega_1(\bar{\beta}), \omega_2(\bar{\beta}))$, for $\bar{\beta}\in \bar{S}$, the map from $M_n(\C)\times M_n(\C)$ to $M_n(\C)\times M_n(\C)$ given by
				$$(b_1,b_2) \mapsto \begin{pmatrix}
					-(w_1+w_2-\beta)^{-1}(b_1+b_2)(w_1+w_2-\beta)^{-1}+\int_\R(w_1-t\gamma_1)^{-1}b_1(w_1-t\gamma_1)^{-1}\mu_{c}(dt)\\
					-(w_1+w_2-\beta)^{-1}(b_1+b_2)(w_1+w_2-\beta)^{-1}+\int_\R(w_2-t\gamma_2)^{-1}b_2(w_2-t\gamma_2)^{-1}\mu_{d}(dt)
				\end{pmatrix}$$
			is invertible.
			\end{enumerate}
			\end{definition}
		\begin{remark}
			We will take $S = \{z\cdot e_{1,1}-\gamma_0: z = x+iy, x\in I, y\in(0, \eta]\}$ for a fixed interval and $V=\R$ in Section \ref{MainResult} in which we apply the linearization technique to polynomials.
			\end{remark}
		
		Given two probability measures $\mu$ and $\nu$, we use the L\'evy distance
		\[d_L(\mu, \nu) = \inf\{s\geq 0: F_\nu(x-s)-s\leq F_\mu(x)\leq F_\nu(x+s)+s\textrm{ for all } x\in\R\}\]
		to measure the distance between $\mu$ and $\nu$, where $F_\mu$ and $F_\nu$ are the cumulative distribution functions of $\mu$ and $\nu$ respectively. A sequence $\mu_n$ of probability measures converges weakly to $\mu$, if and only if $d_L(\mu_n, \mu)\to 0$ (See Theorem III.1.2 on page 314 and Exercise III.1.4 on page 316 in \cite{ShuryaevBook})
				
		We set $\textrm{LD}(c_N, d_N) := \max\{d_L(\mu_{c_N},\mu_c), d_L(\mu_{d_N},\mu_d)\}$. Note that even though $d_N$ is a random matrix, the law $\mu_{d_N}$ is deterministic, since $d_N$ is a unitary conjugate of the deterministic self-adjoint diagonal matrix $D_N$. In our random matrix model, $(c_N, d_N)\to (c,d)$ strongly in distribution, so, in particular, $\textrm{LD}(c_N, d_N)\to0$.  Denote by $\varepsilon_1^{(N)}(\beta)=(\omega_1(\beta)-\gamma_1\tensor c)^{-1}\delta_1(\beta)$  $\varepsilon_2^{(N)}(\beta) = (\omega_2(\beta)-\gamma_1\tensor d)^{-1}\delta_2(\beta)$where $\delta_1$ is as in \eqref{delta1} and $\delta_2$ is defined analogously, and \[\varepsilon^{(N)}(\beta) = \|(\varepsilon_1^{(N)}(\beta)\|+ \|(\varepsilon_2^{(N)}(\beta)\|.\] 
		
		For convenience, we introduce the notation $\mathscr{R}_\eta := \{x+iy: x\in S, y\in V, 0< y\leq \eta\}$ where $\eta>0$. Although $\mathscr{R}_\eta$ depends on $V$ too, there is no ambiguity because $V$ does not change throughout the proof, or in the application to linearization.
		
		When $X$ and $Y$ are Banach spaces, we denote by $B(X,Y)$ the space of all bounded linear operators from $X$ to $Y$.
		
		To prove the next proposition, we need the following version of \cite[Lemma 2.1]{Kargin2015} to use L\'evy distance to approximate integrals. This version of the lemma only considers integration on a compact interval, instead of on the real line. This small change is due to the fact that we will integrate a continuous but possibly unbounded function on $\R$ against a compactly supported measure; integrating over a compact interval poses no question of integrability. 
		
		Since $(c_N, d_N)\to (c,d)$ strongly in distribution, $\textrm{supp} \:\mu_{c_N}\to \textrm{supp} \:\mu_c$ and $\textrm{supp}\: \mu_{d_N}\to \textrm{supp} \:\mu_d$ in the Hausdorff metric. Consequently, there is an $M>0$ such that for large enough $N$, $\mu_c, \mu_d, \mu_{c_N}$ and $\mu_{d_N}$ are all supported in $[-M, M]$.
		\begin{lemma}
			\label{LevyEst}
			Suppose that $\mu$ and $\nu$ are probability measures supported in $[-M, M]$. Suppose also that $d_L(\mu, \nu)=l$. Assume that $h$ is a $C^1$ real-valued function such that $\int_{-M}^M|h(u)\,du|<\infty$ and $\int_{-M}^M|h'(u)|\,du<\infty$. Then
			\[\int_{-M}^M |h(u)[F_\nu(\eta u)-F_\mu(\eta u)]|\,du\leq cl\max\{1, \eta^{-1}\}\]
			where $c>0$ can be chosen to be \[c = \max\left\{\sup_{\xi\in[-M, M]}2|h(\xi)|, \int_{-M}^M |h'(\xi)|\,d\xi\right\}.\]
			\end{lemma}
		We use Newton's method to estimate how far $\omega_{1}^{(N)}$ and $\omega_{2}^{(N)}$ are from the subordination functions $\omega_1$ and $\omega_2$ in the limit.
		\begin{proposition}
			\label{Newton}
			Assume that $(\mu_c, \mu_d, H)$ is regular on $(S, V)$, and that $\bar{S}$ is compact. Then for some positive $\bar{\varepsilon}$, $\bar{t}$, $\bar{\eta}$, if $\varepsilon:=\varepsilon^{(N)}(\beta)\leq \bar{\varepsilon}$, $l:=\textrm{LD}(c_N, d_N)\leq \bar{l}$, and $\beta\in \mathscr{R}_{\bar{\eta}}$, we have
			$$\max\{\|\omega_1(\beta)-\omega_{1}^{(N)}(\beta)\|,\|\omega_2(\beta)-\omega_{2}^{(N)}(\beta)\|\}=O(\varepsilon+l)$$
			where the constant in the $O$ depends only on $c$, $d$, $\gamma_1$, $\gamma_2$.
			\end{proposition}
		
		\begin{proof}
			Consider the function $F:M_n(\C)\times M_n(\C) \to M_n(\C)\times M_n(\C)$ defined by
			$$F(w_1, w_2) = 
			\begin{pmatrix}
			(w_1+w_2-\beta)^{-1}-G_{c_N}(w_1)-\varepsilon_{1}^{(N)}(\beta)\\
			(w_1+w_2-\beta)^{-1}-G_{d_N}(w_2)-\varepsilon_{2}^{(N)}(\beta)
			\end{pmatrix}.$$
			The function $F$ depends on $\beta$ but we will find bounds that are independent of $\beta\in\mathscr{R}_{\bar{\eta}}$ for some $\bar{\eta}$.
			We make the initial guess $w_0(\beta) = (\omega_1(\beta),\omega_2({\beta}))$ in Newton's method.  To simply our notations, we also use the notation $\omega_1$, $\omega_2$  for the continuous extensions of these functions to $\{x+iy\in \bbH^+(M_n(\C)): x\in \bar{S}, y\in V\}$.
			
			Since $(\mu_c, \mu_d, H)$ is regular on $(S, V)$ in the sense of Definition \ref{reg}, $\omega_1$ and $\omega_2$ are continuous on $\overline{\mathscr{R}}_1$, and $\omega_1(\beta)-t\gamma_1$ is invertible for all $\beta\in \overline{\mathscr{R}}_1$ and $t\in\textrm{supp}\:\mu_c$. Because $(c_N, d_N)\to (c, d)$ strongly in distribution, $\textrm{supp}\:\mu_{c_N}\to \textrm{supp}\:\mu_c$ in the Hausdorff metric. Since invertible matrices form an open set, $\omega_1(\bar{\beta})-t\gamma_1$ is invertible for all $\bar{\beta}\in \bar{S}$ and $t\in\textrm{supp}\:\mu_{c_N}$ for all large enough $N$. Similarly, $\omega_2(\bar{\beta})-t\gamma_2$ is invertible for all $\bar{\beta}\in \bar{S}$ and $t\in\textrm{supp}\:\mu_{d_N}$ for all large enough $N$. Because $M_n(\C)$ is finite dimensional, $\overline{\mathscr{R}}_1$ is compact.

			Since $\textrm{supp}\:\mu_c$ and $\overline{\mathscr{R}}_1$ are compact, we have $\|(\omega_1(\beta)-t\gamma_1)^{-1}\|\leq C$ for some $C$ independent of $\beta\in \overline{\mathscr{R}}_1$ and $t\in\textrm{supp}\:\mu_c$. Choosing a bigger $C$ if necessary, $\|(\omega_1(\beta)-t\gamma_1)^{-1}\|\leq C$ for some $C$ independent of $\beta\in \overline{\mathscr{R}}_1$ and $t\in\textrm{supp}\:\mu_{c_N}$ for all large enough $N$. Similarly, $\|(\omega_2(\beta)-t\gamma_2)^{-1}\|\leq C$ for some  (possibly still larger) $C$ independent of $\beta\in \overline{\mathscr{R}}_1$, $t\in\textrm{supp}\:\mu_d$ and $t\in\textrm{supp}\:\mu_{d_N}$ for large enough $N$. In particular, the entries satisfy
			\begin{equation}
			\label{CoeffBound}
			|((\omega_1(\beta)-t\gamma_1)^{-1})_{ij}|\leq C\textrm{ and } |(\omega_2(\beta)-t\gamma_2)^{-1})_{ij}|\leq C.
			\end{equation}
			
			
			The determinant of $\omega_j(\beta)-t\gamma_j$ is a polynomial in $t$ with degree at most $n$, and each entry of the matrix cofactor of $\omega_j(\beta)-t\gamma_j$ is a polynomial in $t$ with degree at most $n-1$. Each entry of $(\omega_j(\beta)-t\gamma_j)^{-1}$ can then be expressed as a rational function
			\[\frac{P_\beta(t)}{Q_\beta(t)}\]
			for some polynomials $P_\beta$ and $Q_\beta$ of degrees at most $n$ and $n-1$ respectively, and whose coefficients are continuous functions of $\beta$. When $N$ is large enough, by \eqref{CoeffBound}, $Q_\beta(t)$ are bounded away from $0$ for all $t\in\textrm{supp}\,\mu_{c_N}$ and $\beta\in\bar{S}$ when $j=1$, and $t\in\textrm{supp}\,\mu_{d_N}$ and $\beta\in\bar{S}$ when $j=2$. Suppose, for large enough $N$, $\mu_c, \mu_d, \mu_{c_N}$ and $\mu_{d_N}$ are all supported in $[-M, M]$. So, by integration by parts and using the fact that $F_{\mu_c}(M) = 1$, each entry of $\int_\R (\omega_1(\beta)-t\gamma_1)^{-1}\,\mu_c(dt)$ is of the form
			\begin{align*}
			\frac{P_\beta(M)}{Q_\beta(M)} - \int_{-M}^M \frac{Q_\beta(\lambda)P_\beta'(\lambda)-P_\beta(\lambda)Q_\beta'(\lambda)}{Q_\beta(\lambda)^2}\,F_{\mu_c}(\lambda)\,d\lambda
			\end{align*}
			It follows that
			\[G_c(\omega_1(\beta))-G_{c_N}(\omega_1(\beta)) = \int_{-M}^M \frac{Q_\beta(\lambda)P_\beta'(\lambda)-P_\beta(\lambda)Q_\beta'(\lambda)}{Q_\beta(\lambda)^2}\,[F_{\mu_{c_N}}(\lambda)-F_{\mu_c}(\lambda)]\,d\lambda.\]
			By Lemma \ref{LevyEst} and the fact that $Q_\beta(t)$ are bounded away from $0$ for all $t\in\textrm{supp}\,\mu_{c_N}$, $t\in\textrm{supp}\,\mu_c$ and $\beta\in\bar{S}$, there is a constant $c$ independent of $\beta$ such that
			\[\|G_c(\omega_1(\beta))-G_{c_N}(\omega_1(\beta))\|\leq cl.\]
			The similar estimate holds for $d_N$ and $d$ by symmetry. So we have
			\[\left\|\begin{pmatrix}
			G_c(\omega_1(\beta))-G_{c_N}(\omega_1(\beta))\\
			G_d(\omega_2(\beta))-G_{d_N}(\omega_2(\beta))
			\end{pmatrix}\right\|\leq Ml\]
			for some constant $M$.
			
			

			
			Now, by definition of $\varepsilon^{(N)}$, $\|(\varepsilon_{1}^{(N)}(\beta), \varepsilon_{2}^{(N)}(\beta))\|\leq \varepsilon^{(N)}(\beta)$ and we have
			$$\|F(w_0)\| = \left\|\begin{pmatrix}
			G_c(\omega_1(\beta))-G_{c_N}(\omega_1(\beta))-\varepsilon_{1}^{(N)}(\beta)\\
			G_d(\omega_2(\beta))-G_{d_N}(\omega_2(\beta))-\varepsilon_{2}^{(N)}(\beta)\\
			\end{pmatrix}\right\|=O(l+\varepsilon)$$
			where the constant term in the $O$-term is uniform for all $\beta\in \overline{\mathscr{R}}_1$.\\
			
			In order to apply Newton's method, we need to estimate the norm of the inverse of the derivative.
			It is easy to compute that the derivative                                                                                                                                                                                                                                                                                                                                                                                                                                                                                                                                                                                                                                                                                                                                                                                                                                                                                                                                                                                                                                                                                                                                                                                                                                                                                                                                                                                                                                                                                                                                                                                                                                                                                                                                                                                                                                                                                                                                                                                                                                                                                                                                                                                                                                                                                                                                                                                                                                                                                                                                                                                                                                                                                                                                                                                                                                                                                                                                                                                                                                                                                                                                                                                                                                                                                                                                                                                                                                                                                                                                                                                                                                                                                                                                                                                                                                                                                                                                                                                                                                                                                                                                                                                                                                                                                                                                                                                                                                                                                                                                                                                                                                                                                                                                                                                                                                                                                                                                                                                                                                                                                                                                                                                                                                                                                                                                                                                                                                                                                                                                                                                                                                                                                                                                                                                                                                                                                                                                                                                                                                                                                                                                                                                                                                                                                                                                                                                                                                                                                                                                                                                                                                                                                                                                                                                                                                                                                                                                                                                                                                                                                                                                                                                                                                                                                                                                                                                                                                                                                                                                                                                                                                                                                                                                                                                                                                                                                                                                                                                                                                                                                                                                                                                                                                                                                                                                                                                                                                                                                                                                                                                                                                                                                                                                                                                                                                                                                                                                                                                                                                                                                                                                                                                                                                                                                                                                                                                                                                                                                                                                                                                                                                                                                                                                                                                                                                                                                                                  $F'_{w_1,w_2}\in B(M_n(\C)^2,M_n(\C)^2)$ is given by
			$$F'_{w_1,w_2}(b_1,b_2) = \begin{pmatrix}
			-(w_1+w_2-\beta)^{-1}(b_1+b_2)(w_1+w_2-\beta)^{-1}+\int_\R(w_1-t\gamma_1)^{-1}b_1(w_1-t\gamma_1)^{-1}\mu_{c_N}(dt)\\
			-(w_1+w_2-\beta)^{-1}(b_1+b_2)(w_1+w_2-\beta)^{-1}+\int_\R(w_2-t\gamma_2)^{-1}b_2(w_2-t\gamma_2)^{-1}\mu_{d_N}(dt)
			\end{pmatrix}$$
			which is invertible when $(w_1, w_2) = w_0(\bar{\beta})$, for $\bar{\beta}\in \bar{S}$ if $\bar{l}$ is chosen to be small enough so that $d_L(\mu_{c_N}, \mu_c)$ and $d_L(\mu_{d_N}, \mu_d)$ small, by the assumption (B) that $(\mu_c, \mu_d, H)$ is regular, and that invertible operators in the space $B(M_n(\C)^2, M_n(\C)^2)$ form an open set. It is evident that $F'$ is continuous, as a function of $(w_1, w_2)$, on $M_n(\C)\times M_n(\C)$. Since invertible operators on $M_n(\C)^2\to M_n(\C)^2$ form an open set and $\omega_1$, $\omega_2$ are continuous, there is an $\bar{\eta}>0$ such that $F'_{\omega_1(\beta), \omega_2(\beta)}$ is invertible for $\beta\in \overline{\mathscr{R}}_{\bar{\eta}}$. Since $M_n(\C)$ and $V$ are finite dimensional, $\overline{\mathscr{R}}_{\bar{\eta}}$ is compact. It follows that $\|[F'_{\omega_1(\beta), \omega_2(\beta)}]^{-1}\underline{}\|$ is bounded uniformly on $\overline{\mathscr{R}}_{\bar{\eta}}$.\\
			
			We now estimate the norm of the second derivative $F'':M_n(\C)^2\to B(M_n(\C)^2, B(M_n(\C)^2, M_n(\C)^2))$. The first entry of $F''_{w_1, w_2}(h_1, h_2)$is given by
			\begin{align*}
			(F''_{w_1, w_2}(h_1, h_2))_1(b_1, b_2) =& (w_1+w_2-\beta)^{-1}(h_1+h_2)(w_1+w_2-\beta)^{-1}(b_1+b_2)(w_1+w_2-\beta)^{-1}\\
				&+ (w_1+w_2-\beta)^{-1}(b_1+b_2)(w_1+w_2-\beta)^{-1}(h_1+h_2)(w_1+w_2-\beta)^{-1}\\
				&-\int_\R [(w_1-t \gamma_1)^{-1}h_1(w_1-t\gamma_1)^{-1}b_1(w_1-t\gamma_1)^{-1}\\
				&\quad+(w_1-t \gamma_1)^{-1}b_1(w_1-t\gamma_1)^{-1}h_1(w_1-t\gamma_1)^{-1}]\:\mu_{c_N}(dt).
			\end{align*}
			Then the norm of $\|(F''_{\omega_1(\beta), \omega_2(\beta)}(h_1, h_2))_1\|$ is bounded by
			$$\|(F''_{\omega_1(\beta), \omega_2(\beta)}(h_1, h_2))_1\|\leq 2(\|(\omega_1(\beta)+\omega_2(\beta)-\beta)^{-1}\|^3+\sup_{\beta, t}\|(\omega_1(\beta)-t\gamma_1)^{-1}\|^3)\|(h_1, h_2)\|$$
			which shows the norm in the space $L(M_n(\C)^2, L(M_n(\C)^2, M_n(\C)^2))$ of
			$$\|(F''_{\omega_1(\beta), \omega_2(\beta)})_1\|\leq 2\left(\|(\omega_1(\beta)+\omega_2(\beta)-\beta)^{-1}\|^3+\sup_{\beta, t}\|(\omega_1(\beta)-t\gamma_1)^{-1}\|^3\right)\leq 4C^3$$
			by the facts that $\sup_{\beta, t}\|(\omega_1(\beta)-t\gamma_1)^{-1}\|\leq C$ and $\|(\omega_1(\beta)+\omega_2(\beta)-\beta)^{-1}\| = \|G_c(\omega_1(\beta))\|\leq C$.
			A similar computation shows $\|(F''_{w_1, w_2})_2\|\leq 4C^3.$
			Since $\omega_1$ and $\omega_2$ are continuous on $\overline{\mathscr{R}}_1$, $\|F''_{\omega_1(\beta), \omega_2(\beta)}\|$ is bounded uniformly, say by $A$, in $\overline{\mathscr{R}}_{\bar{\eta}}$.
			
			We take $C_0 = \sup_{\beta\in \overline{\mathscr{R}}_{\bar{\eta}}}\|F'(w_0(\beta))\|$. Since $\|[F'(w_0(\beta))^{-1}]F(w_0(\beta)\|\leq C_0\|F(w_0(\beta))\|$, we take $\delta_0 =  C_0\sup_{\beta}\|F(w_0(\beta))\|$ which is $O(\varepsilon+l)$, with the constant in the $O$-term uniform in $\beta\in \overline{\mathscr{R}}_{\bar{\eta}}$. We can take small enough $\bar{l}$ and $\bar{\varepsilon}$ such that if $l\leq \bar{l}$ and $\varepsilon(\beta)\leq \bar{\varepsilon}$, then $h_0 = C_0\delta_0A<\frac{1}{2}$ so that Newton's method guarantees the unique solution is in a neighborhood of $w_0$. Newton's method indeed guarantees the solution $(w_1, w_2)$ satisfying
			$$\|w_1-\omega_1(\beta)\|=O(\varepsilon+l)\quad\textrm{ and }\quad\|w_2-\omega_2(\beta)\|=O(\varepsilon+l).$$
			
			The functions $\omega_{1}^{(N)}(\beta)$ and $\omega_{2}^{(N)}(\beta)$ is the zero of $F$, $F(w_1, w_2) = 0$. Also, \eqref{remainder} shows that $\omega_{1}^{(N)}(\beta)\to \omega_1(\beta)$ and $\omega_{2}^{(N)}(\beta)\to\omega_2(\beta)$ as $N\to \infty$. If follows that for small $\bar{l}$ and $\bar{\varepsilon}$, the solution of $F(w_1, w_2)=0$ found by Newton's method is the pair $(\omega_{1}^{(N)}(\beta), \omega_{2}^{(N)}(\beta))$ which proves for all $\beta\in \mathscr{R}_{\bar{\eta}}$ (but not $\overline{\mathscr{R}}_{\bar{\eta}}$),
			$$\|\omega_{1}^{(N)}(\beta)-\omega_1(\beta)\|=O(\varepsilon+l)\quad\textrm{ and }\quad\|\omega_{2}^{(N)}(\beta)-\omega_2(\beta)\|=O(\varepsilon+l).$$
			\end{proof}
		
		\subsection{Main Results}
		\label{MainResult}
			\subsubsection{A Local Limit Theorem for Polynomials}
			Now we move on to estimating the (scalar) Cauchy transform of a polynomial $P(c_N, d_N)$ in $c_N$ and $d_N$. We recall that $(c_N, d_N)\to (c,d)$ strongly in distribution. The polynomial $P$ possesses a linearization as in Section \ref{Linear} which gives a linear polynomial with coefficients in $M_n(\C)$
			$$L_N = \gamma_0+\gamma_1\tensor c_N+\gamma_2\tensor d_N.$$
			We apply the results from the preceding section to $H_N = \gamma_1\tensor c_N + \gamma_2\tensor d_N$ and its large-$N$ limit $H = \gamma_1\tensor c+\gamma_2\tensor d$.

			\begin{lemma}
				\label{ExpEst}
				Assume that $(\mu_c, \mu_d, H)$ is regular on $(S, V)$, and that $\bar{S}$ is compact. Then for some sufficiently small $\varepsilon = \varepsilon(\beta)$, $l = LD(c_N, d_N)$, and $\eta>0$, we have
				$$\|\E G_{H_N}(\beta)- G_{H}(\beta)\| = O(\varepsilon+l)$$
				where $\Re\beta\in S$, $\Im\beta\in V$, $\Im\beta>0$, $\|(\Im\beta)^{-1}\|=\frac{1}{\eta}$.
				\end{lemma}
			\begin{proof}
				Because $\E G_{H_N}(\beta) = (\omega_{1}^{(N)}(\beta)+\omega_{2}^{(N)}(\beta)-\beta)^{-1}$ and $G_{H}(\beta) = (\omega_1(\beta)+\omega_2(\beta)-\beta)^{-1}$, we see that
				\begin{align*}
				&\E[M_n(\tr)(R_N(\beta))]-M_n(\tau)(R(\beta))\\
				=& (\omega_{1}^{(N)}(\beta)+\omega_{2}^{(N)}(\beta)-\beta)^{-1}[(\omega_1-\omega_{1}^{(N)})(\beta)+(\omega_2-\omega_{2}^{(N)})(\beta)](\omega_1(\beta)+\omega_2(\beta)-\beta)^{-1}.
				\end{align*}
				By the subordination relation
				$$(\omega_1(\beta)+\omega_2(\beta)-\beta)^{-1} = \int_\R (\omega_2(\beta)-t\gamma_2)^{-1}\;\mu_{d}(dt),$$
				$(\mu_c, \mu_d, H)$ being regular implies that $(\omega_1(\beta)+\omega_2(\beta)-\beta)^{-1}$ is uniformly bounded for small enough $\eta>0$, $\Re\beta\in S$, $\Im\beta\in V$, $\Im\beta>0$, $\|(\Im\beta)^{-1}\|=\frac{1}{\eta}$.
				Proposition \ref{Newton} shows that $\omega_{1}^{(N)}$ and $\omega_{2}^{(N)}$ are close to $\omega_1$ and $\omega_2$ respectively, so when $N$ large, $(\omega_{1}^{(N)}(\beta)+\omega_{2}^{(N)}(\beta)-\beta)^{-1}$ is also bounded. The norm of the middle term $\|(\omega_c-\omega_{c_N})(\beta)+(\omega_d-\omega_{d_N})(\beta)\|$ is $O(\varepsilon+ l)$ by the conclusion of Proposition \ref{Newton}, provided that $\eta$, $\varepsilon$, $l$ are small enough.
				\end{proof}
			
			We now consider $\beta$ of the particular form $z e_{1,1}-\gamma_0+i\eta$, $z\in\bbH^+(\C)$. We first note that for these $\beta$, $\Im\beta =\textrm{diag}(\Im z+\eta,\eta\ldots,\eta)$. We are interested in the local behavior of the polynomial. So, we consider $z\in \Pi_{I, \kappa} = I+i(0, \kappa]$, a rectangle on the complex plane. This means when we apply Proposition \ref{Newton}, we will assume $H$ is regular in $(\Pi_{I, \kappa}\cdot e_{1,1}-\gamma_0, \R\cdot I_n)$, where 
			\[\Pi_{I, \kappa}\cdot e_{1,1}-\gamma_0=\{z\cdot e_{1,1}-\gamma_0: z\in \Pi_{I, \kappa}\}\]
			and $I_n$ is the $n\times n$ identity matrix.

			\begin{lemma}
				\label{supbound}
				For some positive constants $c_1$ and $c''$ which may depend on $|I|$ and $\kappa$ and for all $\delta>0$,
				$$\P\left(\sup\|G_{H_N}(\beta)-\E G_{H_N}(\beta)\|>\delta\right)\leq \exp\left(\frac{-c'' N\eta^4\delta^2}{\|\gamma_2|\|^2\|d_N\|^2}\right)$$
				where the supremum is taken over $\beta\in \Pi_{I, \kappa}\cdot e_{1,1}-\gamma_0+i\eta$, provided that $N\geq \frac{-c_1\log(\eta^4\delta^2)}{\eta^4\delta^2}$ for some $c_1$ large enough.
				\end{lemma}
			\begin{proof}
				It is easy to compute $\|G_{H_N}'(\beta)\|\leq \|(\Im\beta)^{-1}\|^2$ and $\| \E G_{H_N}'(\beta)\|\leq \|(\Im\beta)^{-1}\|^2$. 
				
				We create an $\frac{\eta^2\delta}{4}$-net on $\Pi_{I, \kappa}$, which requires us $O(|I|\cdot \kappa)$ number of points. If $\|G_{H_N}(\beta_0)-\E G_{H_N}(\beta_0)\|\leq \delta/2$ for every point $\beta_0$ on the net, then given any $\beta=z e_{1,1}-\gamma_0+i\eta$, choose $\beta_0$ such that $\|\beta-\beta_0\|\leq \frac{1}{\sqrt{2}}\frac{\eta^2\delta}{4}<\frac{\eta^2\delta}{4}$. Then
				\begin{align*}
				\|G_{H_N}(\beta)-\E G_{H_N}(\beta)\|\leq &\|G_{H_N}(\beta)-G_{H_N}(\beta_0)\|+\|G_{H_N}(\beta_0)-\E G_{H_N}(\beta_0)\|+\|\E G_{H_N}(\beta_0)-\E G_{H_N}(\beta)\|\\
				\leq & \|(\Im\beta_0)^{-1}\|^2\|\beta-\beta_0\|+\frac{\delta}{2}+\|(\Im\beta_0)^{-1}\|^2\|\beta-\beta_0\|\leq \delta.				
				\end{align*}
				Recall, by the concentration inequality, for all $\alpha\in\left(0,\frac{1}{2}\right)$, for every $\beta\in \bbH^+(M_n(\C))$, $\|\Im\beta^{-1}\|=\frac{1}{\eta}$, we have
				$$\P\left(\|G_{H_N}(\beta)-\E G_{H_N}(\beta)\|>\delta\right)\leq 2\exp\left(\frac{-c N\eta^4\delta^2}{\|\gamma_2\|^2\|d_N\|^2}\right).$$
				Since there are $O\left(|I|\cdot \kappa\right)$ number of points on the net,
				\begin{align*}
				\P\left(\sup_{\Pi_{I,\kappa}\cdot e_{1,1}-\gamma_0+i\eta}\|G_{H_N}(\beta)-\E G_{H_N}(\beta)\|>\delta\right)\leq& \frac{C'|I|\cdot\kappa}{\eta^4\delta^2}\exp\left(\frac{-c N\eta^4\delta^2}{\|\gamma_2|\|^2\|d_N\|^2}\right)\\
				\leq& \exp\left(\frac{-c N\eta^4\delta^2}{\|\gamma_2|\|^2\|d_N\|^2}+\log \left(\frac{C'|I|\cdot\kappa}{\eta^4\delta^2}\right)\right)\\
				\leq& \exp\left(\frac{-c'' N\eta^4\delta^2}{\|\gamma_2|\|^2\|d_N\|^2}\right)
				\end{align*}
				if $N\geq \frac{-c_1\log(\eta^4\delta^2)}{\eta^4\delta^2}$ for some $c_1$ large enough.
				\end{proof}
			
				Let $\delta^2 = \frac{\log N}{N\eta^4}$. Then $N\geq \frac{-c_1\log(\eta^4\delta^2)}{\eta^4\delta^2}=\frac{-c_1\log((\log N)/N)}{(\log N)/N}$ when $N$ large. If $\eta = N^{-1/12}\sqrt{\log N}$, then Lemma \ref{supbound} shows that, as long as $N$ is large enough,
				$$\P\left(\sup_{\Pi_{I,\kappa}\cdot e_{1,1}-\gamma_0+i\eta}\|G_{H_N}(\beta)-\E G_{H_N}(\beta)\|>\frac{1}{N^{\frac{1}{3}}\sqrt{\log N}}\right)\leq \exp\left(-c \log N\right)$$
				for some constant $c>0$.
				
				Given any $\varepsilon>0$, by \eqref{Eerror} and Proposition \ref{Impart}, $\varepsilon^{(N)}(\beta)<\varepsilon$ when $N$ large. Lemma \ref{ExpEst} shows that
				$$\|\E G_{H_N}(\beta)-G_{H}(\beta)\|\leq \varepsilon + Cl.$$
				We have proved the following.
				\begin{proposition}
					\label{Conc}
					Assume that the triple $(\mu_c, \mu_d, H)$ is regular in $(\Pi_{I, \kappa}\cdot e_{1,1}-\gamma_0, \R\cdot I_n)$. Assume also that $l:=LD(c_N, d_N)\leq \bar{l}$ where $\bar{l}$ is chosen as in Proposition \ref{Newton}. Let $\eta = N^{-1/12}\sqrt{\log N}$. Then for some positive constants $C$ and $c$, and for every $\varepsilon>0$,
					$$\P\left(\sup_{\Pi_{I,\kappa}\cdot e_{1,1}-\gamma_0+i\eta}\|G_{H_N}(\beta)- G_{H}(\beta)\|>\varepsilon+Cl\right)\leq \exp\left(-c\log N\right)$$
					for all sufficiently large $N$.
					\end{proposition}
				We denote the Hilbert-Schmidt norm by $\|\cdot\|_{HS}$. The function $X\mapsto \|X\|_{HS}$  is clearly Lipschitz, with Lipschitz constant $1$, in the Hilbert-Schmidt norm. By \cite[Corollary 4.4.28]{Book2010} and \cite[Lemma 8.3]{BBC2018}, for each $\beta$ with $\|\Im\beta\|=\eta$, 
				\begin{equation}
				\label{ConcIneq}
				\E\left(\|R_N(\beta)\|_{HS}-\E[\|R_N(\beta)\|_{HS}]\geq \varepsilon \right)\leq 2\exp\left(\frac{-\varepsilon^2 N\eta^4}{16\|\gamma_2\|^2\|D_N\|^2}\right).
				\end{equation}
			\begin{lemma}
				\label{Bound}
				Let $\eta = N^{-\frac{1}{12}}(\log N)^\delta$, $\delta>0$ and $\beta$ is of the form $z e_{1,1}-\gamma_0+i\eta$, where $z\in \Pi_{I, \kappa}$. If $(\mu_c, \mu_d, H)$ is regular in $(\Pi_{I, \kappa}\cdot e_{1,1}-\gamma_0, \R\cdot I_n)$, then the (random) matrix norm $\|R_N(\beta)\|$ of $R_N(\beta)$ is bounded uniformly with large probability. That is, there exists a $C>0$ such that
				$$\P\left(\sup_{\Pi_{I,\kappa}\cdot e_{1,1}-\gamma_0+i\eta}\|R_N(\beta)\|_{HS}>C\right)\leq \exp\left(\frac{-c'' N\eta^4\delta^2}{\|\gamma_2|\|^2\|d_N\|^2}\right)$$
				for all large enough $N$.
				\end{lemma}
			\begin{proof}
				
				Recall that, by Proposition \ref{Newton}, and Theorem \ref{FinEst} (2) that $\varepsilon^{(N)}(\beta) = O\left(\frac{1}{N^{\frac{1}{2}}\eta^5}\right)$,
				\begin{equation}\label{wdist}\|\omega_d(\beta)-\omega_{d_N}(\beta)\| = O\left(\frac{1}{N^{\frac{1}{2}}\eta^5}+\textrm{LD}(c_N, d_N)\right)
				\end{equation}
				when $N$ is large enough, so that $\textrm{LD}(c_N, d_N)$ is small enough. We also recall that \[\textrm{LD}(c_N, d_N) = \max\{d_L(\mu_{c_N},\mu_c),d_L(\mu_{d_N},\mu_d)\}.\]
				
				If $N^{\frac{1}{2}}\eta^4$ is large enough, by Theorem \ref{FinEst}, $\|\E R_N(\beta)-(\omega_{1}^{(N)}(\beta)-\gamma_1\tensor c_N)^{-1}\|=O\left(\frac{1}{N^{\frac{1}{2}}\eta^5}\right)$.
				We notice that $$(\omega_{1}^{(N)}(\beta)-\gamma_1\tensor c_N)^{-1} = \int_\R (\omega_{1}^{(N)}(\beta)-t\gamma_1)^{-1}\;E_N(dt)$$
				where $E_N$ is the (discrete) spectral measure for $c_N$. The support of $E_N$ converges to $\textrm{supp}\:\mu_c$ in Hausdorff distance by strong convergence in distribution. Using equation \eqref{wdist} and Proposition \ref{Newton}, when $N$ is large enough, regularity on $(\Pi_{I,\kappa}\cdot e_{1,1}-\gamma_0, \R\cdot I_n)$ implies that $(\omega_{1}^{(N)}(\beta)-\gamma_1\tensor c_N)^{-1}$ is bounded uniformly. We conclude that $\|\E R_N(\beta)\|<C$ for some $C>0$, for all $\beta = ze_{1,1}-\gamma_0+i\eta$ where $z\in \Pi_{I,\kappa}$.

				By \eqref{ConcIneq}, for a particular $\beta_0$, with probability at least $1-2\exp\left(\frac{-N\varepsilon^2\eta^4}{16 \|\gamma_2\|^2\|d_N\|^2}\right)$, we have
				$$\|R_N(\beta_0)\|_{HS}\leq \varepsilon +\E\|R_N(\beta_0)\|_{HS}.$$
				
				We create an $\frac{\eta^2\varepsilon}{4}$-net on $\Pi_{I, \kappa}$, which requires $O(|I|\cdot\kappa)$ number of points. Given any $\beta$, there exists $\beta_0$ such that $\|\beta-\beta_0\|\leq \frac{\eta^2\varepsilon}{4}$ and $\|R_N(\beta_0)\|\leq \frac{\varepsilon}{2} + \E\|R_N(\beta_0)\|$. Then using the inequality $\|AB\|_{HS} \leq \|A\|\|B\|_{HS}$ we have the estimates
				$$\big|\|R_N(\beta)\|_{HS}-\|R_N(\beta_0)\|_{HS}\big|\leq\|R_N(\beta)-R_N(\beta_0)\|_{HS}\leq \|\beta-\beta_0\|_{HS}\|(\Im\beta)^{-1}\|\|(\Im\beta_0)^{-1}\|\leq \frac{\varepsilon}{4}.$$
				Similarly,
				$$\big|\|\E R_N(\beta)\|_{HS}-\|\E R_N(\beta_0)\|_{HS}\big|\leq\|\E R_N(\beta)-\E R_N(\beta_0)\|_{HS}\leq  \|\beta-\beta_0\|_{HS}\|(\Im\beta)^{-1}\|\|(\Im\beta_0)^{-1}\|\leq \frac{\varepsilon}{4}.$$
				It follows that, when $N$ is sufficiently large,
				$$\P\left(\sup_{\Pi_{I,\kappa}\cdot e_{1,1}-\gamma_0+i\eta}\|R_N(\beta)\|_{HS}>\varepsilon+\E\|R_N(\beta\|\right)\leq \frac{C'|I|\cdot\kappa}{\eta^4\delta^2}\exp\left(\frac{-c N\eta^4\delta^2}{\|\gamma_2|\|^2\|d_N\|^2}\right)\leq \exp\left(\frac{-c'' N\eta^4\delta^2}{\|\gamma_2|\|^2\|d_N\|^2}\right)$$
				for some constants $C', c, c''>0$. Since $\|\E R_N(\beta)\|<C$, we conclude the proof.
				\end{proof}

				We use the regularity at the boundary of $M_n(\tr)[(ze_{1,1}-L_N+i\eta)^{-1}]$ to get estimates of $M_n(\tr)[(ze_{1,1}-L_N)^{-1}]$; the latter is what we actually want to estimate because the $(1,1)$-entry of $M_n(\tr)[(ze_{1,1}-L_N)^{-1}]$ will give us the Cauchy transform of $P(c_N, d_N)$.
				
				We have 
				\begin{align*}
					&\;\|M_n(\tr)[(ze_{1,1}-L_N)^{-1}] - M_n(\tau)[(ze_{1,1}-L)^{-1}]\|_{HS}\\
					\leq &\;\|M_n(\tr)[(ze_{1,1}-L_N)^{-1}]-M_n(\tr)[(ze_{1,1}-L_N+i\eta)^{-1}]\|_{HS}\\
					&\;+\|M_n(\tau)[(ze_{1,1}-L)^{-1}]-M_n(\tau)[(ze_{1,1}-L+i\eta)^{-1}]\|_{HS}\\
					&\;+\|M_n(\tr)[(z e_{1,1}-L_N+i\eta)^{-1}] -M_n(\tau)[(z e_{1,1}-L+i\eta)^{-1}]\|_{HS}\\
					\leq &\;\eta\|(ze_{1,1}-L_N)^{-1}\|\|(ze_{1,1}-L_N+i\eta)^{-1}\|_{HS}+\eta\|(ze_{1,1}-L)^{-1}\|\|(ze_{1,1}-L+i\eta)^{-1}\|_{HS}\\
					&+\|M_n(\tr)[(z e_{1,1}-L_N+i\eta)^{-1}] -M_n(\tau)[(z e_{1,1}-L+i\eta)^{-1}]\|_{HS}.
				\end{align*}
				We take $z = x+i\eta^\alpha$, where $x\in I$ and $0<\alpha<1$. The quantities $\|(ze_{1,1}-L_N+i\eta)^{-1}\|$, $\|(ze_{1,1}-L+i\eta)^{-1}\|$ are bounded with large probability, by Lemma \ref{Bound}. Since $(\mu_c, \mu_d, H)$ is regular on $\Pi_{I, \varepsilon}-\gamma_0$, and by Lemma \ref{Linv}, 
				\begin{align}
				\label{Boundary}
				&\|M_n(\tr)[(ze_{1,1}-L_N)^{-1}] - M_n(\tau)[(ze_{1,1}-L)^{-1}]\|_{HS}\\
				\leq &M\eta^{1-\alpha}+\|M_n(\tr)[(z e_{1,1}-L_N+i\eta)^{-1}] -M_n(\tau)[(z e_{1,1}-L+i\eta)^{-1}]\|_{HS}\\
				\leq &M\eta^{1-\alpha}+\sqrt{n}\|M_n(\tr)[(z e_{1,1}-L_N+i\eta)^{-1}] -M_n(\tau)[(z e_{1,1}-L+i\eta)^{-1}]\|
				\end{align}
				for some constant $M>0$. We note that we used the fact that $\|A\|_{HS}^2 = n\cdot\tr(A^*A)\leq n\|A^*A\| \leq n\|A\|^2$ for any matrix $A$. The preceding observation leads to the following result, by Proposition \ref{Conc}.
		\begin{proposition}
			\label{Cauchy}
		Assume that the triple $(\mu_c, \mu_d, H)$ is regular on $(\Pi_{I, \kappa}\cdot e_{1,1}-\gamma_0, \R\cdot I_n)$ for some $\kappa>0$. Assume also that $l:=LD(c_N, d_N)\leq \bar{l}$ where $\bar{l}$ is chosen such that the conclusion of Proposition \ref{Newton} holds. Let $\eta = N^{-1/12}\sqrt{\log N}$. Then for some positive constants $M$, $C$ and $c$, and for every $0<\alpha<1$, $\varepsilon>0$,
		$$\P\left(\sup_{I+i\sqrt{\eta}}\|M_n(\tr)[(ze_{1,1}-L_N)^{-1}] - M_n(\tau)[(ze_{1,1}-L)^{-1}]\|_{HS}>M\eta^{1-\alpha}+\varepsilon+Cl\right)\leq \exp\left(-c\log N\right)$$
		for all sufficiently large $N$.
		\end{proposition}
				
		Being regular on $(\Pi_{I, \kappa}\cdot e_{1,1}-\gamma_0, \R\cdot I_n)$ ensures that $\mu_{P(c,d)}$ is absolutely continuous on the interval $I$; this follows from two results. In \cite[Lemma 1.10]{Belinschi2006} Belinschi showed that if we denote $\mu^{\textrm{sc}}$ as the singular continuous part of $\mu$, then for $\mu^{\textrm{sc}}$-almost all $x\in\R$, the nontangential limit of the imaginary part of the Cauchy transform is infinite. Bercovici and Voiculescu \cite{BD1998} proved that $(z-x)G_\mu(z)\to\mu(\{x\})$ as $z\to x$ nontangentially. 
				
		The main theorem of this section is a local limit theorem. The proof is almost identical to the ones given by \cite[Corollary 4.2]{ESY2009} and \cite[Theorem 1.4]{Kargin2015}.
		\begin{theorem}
			\label{LLT}
			Suppose that $\mu_{c_N}\to \mu_c$ and $\mu_{d_N}\to\mu_d$ weakly. Using the same notation as above: $H$ is the nonconstant part of the linearization $L$ of $P$. Suppose $(\mu_c, \mu_d, H)$ is regular on $(\Pi_{I, \kappa}\cdot e_{1,1}-\gamma_0, \R\cdot I_n)$. Let $\rho$ be the density of $P(c, d)$, that is the absolutely continuous part of $\mu_{P(c,d)}$. Assume that $\eta_N^* = N^{-\frac{1}{12}}\log N$ and fix any $\alpha\in(0,1)$. Then for every $x\in I$,
			$$\frac{M_{(\eta_N^*)^\alpha}(x)}{2(\eta_N^*)^\alpha N}\to \rho(x)$$
			in probability, where $M_{(\eta_N^*)^\alpha}(x)$ is the number of eigenvalues in the interval centered at $x$ of length $2(\eta_N^*)^\alpha$.
			\end{theorem}
		\begin{proof}
			In this proof, we will drop the subscript of $\eta_N^*$, and write $\eta^*$ instead.
			Fix any $\alpha\in(0,1)$. Let $\eta = c N^{-1/12}$ and $c$ is sufficiently large, and let $\eta^* = B^\frac{1}{\alpha}\eta$. We write $I^* = [x-(\eta^*)^\alpha, x+(\eta^*)^\alpha]$. Let
			\begin{align*}
			R(\lambda) := &\frac{1}{\pi}\int_{x-(\eta^*)^\alpha}^{x+(\eta^*)^\alpha}\frac{\eta^\alpha}{(t-\lambda)^2+(\eta^\alpha)^2}\;dt\\
			=&\frac{1}{\pi}\left(\arctan\left(\frac{x-\lambda}{\eta^\alpha}+B\right)-\arctan\left(\frac{x-\lambda}{\eta^\alpha}-B\right)\right).
			\end{align*}
			Then $R = \1_{I^*}+T_1+T_2+T_3$ where $\1_{I^*}$ is the indicator function of $I^*$ and functions $T_1$, $T_2$ and $T_3$ satisfy the following properties:
			\begin{equation*}
			\begin{array}{ll}
			|T_1|\leq \frac{c}{\sqrt{B}}\qquad\qquad\qquad&\textrm{supp}(T_1)\subseteq I_1= [x-2(\eta^*)^\alpha, x+2(\eta^*)^\alpha];\\
			|T_2|\leq 1,&\textrm{supp}(T_2)\subseteq J_1\cup J_2;\\
			|T_3|\leq \frac{C\eta^\alpha(\eta^*)^\alpha}{(\lambda-x)^2+(\eta^*)^{2\alpha}}&\textrm{supp}(T_3)\subseteq I_1^c.
			\end{array}
			\end{equation*}
			The $J_1$ and $J_2$ above are the intervals of length $\sqrt{B}\eta^\alpha$ with midpoints $x-(\eta^*)^\alpha$ and $x+(\eta^*)^\alpha$ respectively. Note that
			\begin{align*}
				\frac{M_{(\eta^*)^\alpha}(x)}{2(\eta^*)^\alpha N} &= \frac{1}{2(\eta^*)^\alpha}\int_\R \1_{I^*}(\lambda)\mu_{P_N}(d\lambda)\\
				&= \frac{1}{2(\eta^*)^\alpha}\int_\R R(\lambda) \mu_{P_N}(d\lambda)-\frac{1}{2(\eta^*)^\alpha}\int_\R (T_1+T_2+T_3)\mu_{P_N}(d\lambda).
				\end{align*}
			The second integral can be estimated as
			$$\frac{1}{2(\eta^*)^\alpha}\int_\R(T_1+T_2+T_3)\mu_{P_N}(d\lambda)\leq \frac{c}{\sqrt{B}}\frac{M_{I_1}(x)}{2(\eta^*)^\alpha N}+\frac{M_{J_1}(x)+M_{J_2}(x)}{2(\eta^*)^\alpha N} +\frac{C\eta^\alpha}{(\eta^*)^\alpha}\rho_{(\eta^*)^\alpha}^{(N)}(x)$$
			where $M_I$ denotes the number of eigenvalues of $H_N$ in interval $I$, and
			$$\rho_{(\eta^*)^\alpha}^{(N)}(x):= \frac{1}{\pi}\Im G_{P_N}(x+i(\eta^*)^\alpha) =\frac{1}{\pi}\int_\R\frac{(\eta^*)^\alpha}{(x-\lambda)^2+(\eta^*)^{2\alpha}}\mu_{P_N}(d\lambda).$$
			Hence, by using the inequality 
			$$\frac{1}{\pi}\int_\R\frac{(\eta^*)^\alpha}{(x-\lambda)^2+(\eta^*)^{2\alpha}}\mu_{P_N}(d\lambda) \geq \frac{1}{\pi}\int_{x-(\eta^*)^\alpha}^{x+(\eta^*)^\alpha}\frac{(\eta^*)^\alpha}{(x-\lambda)^2+(\eta^*)^{2\alpha}}\mu_{P_N}(d\lambda)\geq \frac{1}{2\pi}\frac{1}{(\eta^*)^\alpha}\frac{M_{(\eta^*)^\alpha}(x)}{N},$$
			we have
			$$\frac{1}{2(\eta^*)^\alpha}\int_\R|T_1+T_2+T_3|\mu_{P_N}(d\lambda)\leq \frac{c}{\sqrt{B}}[\rho_{2(\eta^*)^\alpha}^{(N)}(x)+\rho_{\sqrt{B}\eta^\alpha}^{(N)}(x-(\eta^*)^\alpha)+\rho_{\sqrt{B}\eta^\alpha}^{(N)}(x+(\eta^*)^\alpha)+\rho_{(\eta^*)^\alpha}^{(N)}(x)].$$
			By the linearization property, $\E \,G_{P_N}(x+i(\eta^*)^\alpha) = M_n(\E\tr)[((x+i(\eta^*)^\alpha)e_{1,1}-L(c_N, d_N))^{-1}]_{1,1}$. So, putting $z=x+i(\eta^*)^\alpha$ we look at
			\begin{align*}
			&\big\{M_n(\E\tr)[(ze_{1,1}-L(c_N, d_N))^{-1}] - M_n(\E\tr)[(ze_{1,1}+i\eta-L(c_N, d_N))^{-1}]\big\}\\
			&\qquad\qquad+M_n(\E\tr)[(ze_{1,1}+i\eta-L(c_N, d_N))^{-1}]\\
			=&M_n(\E\tr)\left(i\eta(z e_{1,1}-L(c_N, d_N))^{-1}(ze_{1,1}-L(c_N, d_N)+i\eta)^{-1}\right)+M_n(\E\tr)[(ze_{1,1}+i\eta-L(c_N, d_N))^{-1}\\
			=&M_n(\E\tr)\left(i\eta(z e_{1,1}-L(c_N, d_N))^{-1}(ze_{1,1}-\gamma_0+i\eta-H_N)^{-1}\right)+G_{H_N}(ze_{1,1}-\gamma_0+i\eta)
			\end{align*}
			
			The first term is bounded by $\eta^{1-\alpha}\|(ze_{1,1}-\gamma_0+i\eta-H_N)^{-1}\|$, which is bounded (and indeed goes to $0$ as $\eta\to 0$) by the assumption that $(\mu_c, \mu_d, H)$ is regular and Proposition \ref{Newton}; the argument here is similar to the proof Lemma \ref{ExpEst}. The regularity assumption implies that
			 \begin{align*}
			&G_{H_N}(ze_{1,1}-\gamma_0+i\eta) \\
			=& [G_{H_N}(ze_{1,1}-\gamma_0+i\eta)-G_{c_N}(\omega_{1}^{(N)}(ze_{1,1}-\gamma_0+i\eta))]+G_{c_N}(\omega_{1}^{(N)}(ze_{1,1}-\gamma_0+i\eta))\\
			=&O(N^{-\frac{1}{2}}\eta^{-5})+G_{c_N}(\omega_{1}^{(N)}(ze_{1,1}-\gamma_0+i\eta))\\
			=&o(1)+G_{c_N}(\omega_{1}^{(N)}(ze_{1,1}-\gamma_0+i\eta))
			\end{align*}
			is bounded, so is $\E \,G_{P_N}(x+i(\eta^*)^\alpha)$. Whence, $\frac{1}{2(\eta^*)^\alpha}\int_\R(T_1+T_2+T_3)\mu_{P_N}(d\lambda) = O (B^{-\frac{1}{2}})$.
			
			Now the main term $\frac{1}{2(\eta^*)^{\alpha}}\int_\R R(\lambda) \mu_{P_N}(d\lambda)$ can be written as
			\begin{align*}
				\frac{1}{2(\eta^*)^{\alpha}}\int_\R R(\lambda) \mu_{P_N}(d\lambda) =& \frac{1}{\pi}\int_{I^*}\Im G_{P}(t+i(\eta^*)^{\alpha})\;dt\\
				+& \frac{1}{2(\eta^*)^{\alpha}}\int_{I^*} \frac{1}{\pi}\Im[G_{P_N}(t+i(\eta^*)^{\alpha})-G_{P}(t+i(\eta^*)^{\alpha})]\;dt.
			\end{align*}
			The first part converges to $\rho(x)$ because of the assumption that $(\mu_c, \mu_d, H)$ is regular at $x$. For the second term, it goes to zero in probability by Proposition \ref{Cauchy} because we assume both $\|\mu_{c_N}-\mu_c\|\to 0$ and $\|\mu_{d_N}-\mu_d\|\to 0$ as $N\to \infty$. 
			\end{proof}

\subsubsection{Delocalization of Eigenvectors}

Delocalization of eigenvectors was proved in \cite{ESY2009} for the general Wigner random matrices. Kargin \cite{Kargin2015} considered delocalization of different order.

\begin{definition}[\cite{Kargin2015}]
We say that the eigenvectors $v_1^{(N)},\ldots,v_{k_N}^{(N)}$ with corresponding eigenvalues in $I$ of a sequence of matrices are delocalized at length $N^{-\theta}$ in the interval $I$ if there exists $\delta>0$ such that 
$$\P\left(\max_{j=1,\cdots k_N}\max_{i=1,\ldots,N}|v_j^{(N)}(i)|>N^{-\theta}\log N\right)\leq \exp(-N^{-\delta})$$
for all sufficiently large $N$.
\end{definition}

Kargin proved the delocalization of eigenvectors for a sum of random matrices of the form $C_N + U_N^*D_NU_N$, using smoothness assumption. The next theorem shows that if we have the regularity condition for the $M_n(\C)$-valued matrices arised from linearization of polynomials, we can prove the delocalization of eigenvectors for polynomials $P(c_N, d_N)$ of two random matrices.

\begin{theorem} 
	Suppose that $(c_N, d_N)\to (c,d)$ strongly in distribution. Denote $H$ as the nonconstant part of the linearization $L$ of $P$. Suppose also that $(\mu_c, \mu_d, H)$ is regular on $(\Pi_{I, \kappa}\cdot e_{1,1}-\gamma_0, \R\cdot I_n)$, where $I$ is a compact interval on $\R$ and $\kappa>0$. Then, given any $\alpha\in(0,1)$, for all sufficiently large $N$, the eigenvectors of $P(c_N, d_N)$ are delocalized at length $N^{-\frac{\alpha}{12}}$ in $I$.
	\end{theorem}
\begin{proof}
	It suffices to show that, given any $v_m^{(N)}$ of $P(c_N, d_N)$,
	$$\P(|v_m^{(N)}(i)|^2>M' N^{\frac{-\alpha}{12}}\log N)\leq 2\exp\left(-K'N\varepsilon^2\right)$$
	for some positive constant $K'$ because $k_N\leq N$; we are taking the intersection of at most $N^2$ events of the form $\{|v_m(k)|^2>M' N^{\frac{-\alpha}{12}}\log N\}$ in the probability in the statement of delocalization of eigenvectors.
	
	Let $\eta = N^{-\frac{1}{12}}(\log N)^{1/\alpha}$. We write $\beta = ze_{1,1}-\gamma_0+i\eta$ for $z\in \Pi_{I,\kappa}$. By Proposition \ref{Impart}, $\omega_{1}^{(N)}(\beta)\in \bbH^+(M_n(\C))$ when $N$ is large.
	
	Again, as in \cite[Proof of Proposition 8.4]{BBC2018}, the concentration inequality implies there is a constant $K>0$ such that for all $i=1,2,\ldots n$,
	$$\P(|(R_N(\beta)-\E R_N(\beta))_{i,i}|>\varepsilon)\leq 2\exp\left(-KN\varepsilon^2\right)$$
	for all $\varepsilon>0$. The $(1,1)$-block of $R_N(z e_{1,1}-\gamma_0)$ is the resolvent $(z-P(c_N, d_N))^{-1}$; the $(i,i)$-entry of $R_N(ze_{1,1}-\gamma_0)$ is the $(i,i)$-entry of the resolvent  $(z-P(c_N, d_N))^{-1}$.
	
	Now, let $z=x+i\eta^\alpha$ where $x\in I$ and $\alpha\in(0,1)$. We have
	\begin{align*}
	 (ze_{1,1}-L)^{-1} =& (ze_{1,1}-L)^{-1}-(ze_{1,1}+i\eta-L)^{-1} + [R(\beta)-\E(R(\beta))]+\E[R(\beta)]\\
	 =&\: i\eta(ze_{1,1}-\gamma_0-\gamma_1\tensor c_N-\gamma_2\tensor d_N)^{-1}(ze_{1,1}-\gamma_0+i\eta-\gamma_1\tensor c_N-\gamma_2\tensor d_N)^{-1}\\
	 &+[R_N(\beta)-\E R_N(\beta)]+\E R_N(\beta).
	\end{align*}
	It follows that, by applying Lemma \ref{Linv} to the first term, and by Lemma \ref{Bound}, there is an $M>0$ such that, with probability at least $2\exp\left(-K'N\varepsilon^2\right)$ for some $K'>0$,
	\begin{equation}
	\label{Resolventk}
	|[(z-P(c_N, d_N))^{-1}]_{i,i}|\leq M\eta^{1-\alpha} + |[R_N(\beta)-\E R_N(\beta)]_{k,k}|+C\leq M\eta^{1-\alpha}+\varepsilon+C.
	\end{equation}

	Let $v_m^{(N)}$, $m=1,\ldots, N$, be the eigenvectors of $P(c_N, d_N)$, with corresponding eigenvalues $\lambda_m$. Since
	$$(z-P(c_N, d_N))^{-1} = \sum_{m=1}^N \frac{v_m v_m^*}{z-\lambda_m},$$
	and so
	$$-\Im[(z-P(c_N, d_N))^{-1}]_{i,i} = \sum_{m=1}^N \frac{\eta^\alpha |v_m(i)|^2}{(x-\lambda_m)^2+\eta^{2\alpha}}$$
	where $v_m(i)$ means the $i$th-component of $v_m$. Taking $x=\lambda_m$ for a particular $\lambda_m\in I$ and using the inequality \eqref{Resolventk} give
	\begin{align*}
	|v_m(i)|^2\leq& -\eta^\alpha\Im[(z-P(c_N, d_N))^{-1}]_{k,k}\\
	\leq &M\eta+  \eta^\alpha|[R_N(\beta)-\E R_N(\beta)]_{k,k}|+C\eta^\alpha\\
	\leq & (M+\varepsilon+C)\eta^\alpha = (M+\varepsilon+C)N^{\frac{-\alpha}{12}}\log N
	\end{align*}
	since $\eta<1$ when $N$ large. 
	It follows that
	$$\P(|v_m(i)|^2>M' N^{\frac{-\alpha}{12}}\log N)\leq 2\exp\left(-K'N\varepsilon^2\right).$$
	This concludes the proof.
	\end{proof}

\section{Acknowledgement}
The author would like to thank Hari Bercovici for useful conversations and discussions; the author would also like to thank his careful reading and valuable advice that improves the manuscript. The author learned the use of linearization from conversations with Tobias Mai and Roland Speicher, when he visited Germany, organized and funded by Todd Kemp and Roland Speicher.

The author would also like to thank the anonymous referee providing useful comments and suggestions to improve the manuscript.


\bibliographystyle{acm}
\bibliography{LocalLimitandDel}

\end{document}